\documentclass[12pt]{amsart}

\usepackage{amsmath}
\usepackage{amssymb}
\usepackage{amstext}
\usepackage{amscd}
\usepackage[frenchb,english]{babel}
\usepackage{enumitem}
\usepackage{greekctr}
\usepackage[matrix,arrow,ps]{xy}
\usepackage{hyperref}

\newtheorem{teo}{Theorem}[section]
\newtheorem{prop}[teo]{Proposition}
\newtheorem{lem}[teo]{Lemma}
\newtheorem{cor}[teo]{Corollary}
\newtheorem{conj}[teo]{Conjecture}

\newtheorem{defi}[teo]{Definition}

\newcommand{\GSp}{{\rm GSp}}
\newcommand{\GL}{{\rm GL}}

\newcommand{\Sh}{{\rm Sh}}

\newcommand{\Supp}{{\rm Supp}}

\newcommand{\Gal}{{\rm Gal}}

\newcommand{\der}{{\rm der}}

\newcommand{\ad}{{\rm ad}}

\newcommand{\CC}{{\mathbb C}}
\newcommand{\RR}{{\mathbb R}}
\newcommand{\ZZ}{{\mathbb Z}}
\newcommand{\QQ}{{\mathbb Q}}

\newcommand{\HH}{{\mathbb H}}

\newcommand{\Hcal}{{\mathcal H}}

\newcommand{\SSS}{{\mathbb S}}
\newcommand{\AAA}{{\mathbb A}}

\newcommand{\lto}{\longrightarrow}

\def\Fh{\mathfrak{h}}
\def\Fu{\mathfrak{u}}

\def\Fm{\mathfrak{m}}
\def\Fz{\mathfrak{z}}

\newcommand{\cF}{\mathcal{F}}
\newcommand{\cA}{{\mathcal A}}

\newcommand{\cZ}{{\mathcal Z}}

\newcommand{\ol}{\overline}

\newcommand{\wt}{\widetilde}

\usepackage{mathtools}
\DeclareMathOperator{\Radu}{{\mathrm{Rad}^{\mathrm{u}}}}
\DeclarePairedDelimiterX{\Nm}[1]{\lVert}{\rVert}{#1}

\title{Inner Galois equidistribution\\ in $S$-Hecke orbits}
\author{ Rodolphe Richard, Andrei Yafaev}
\date{\today}
 \address{Richard, Yafaev: UCL, Department of Mathematics, Gower street, WC1E 6BT, London, UK}
\email{Rodolphe.Richard@normalesup.org, yafaev@math.ucl.ac.uk}

\begin{document}
\setcounter{tocdepth}{1}
\setcounter{secnumdepth}{4}
\maketitle

\tableofcontents

\section{Introduction.}

This paper concerns itself with certain special cases of the Zilber-Pink conjecture on unlikely intersections
in Shimura varieties and some of its natural generalisations.

Let $(G,X)$ be a Shimura datum, let $K$ be a compact open subgroup of $G(\AAA_f)$, and let
\[
\Sh_K(G,X) = G(\QQ)\backslash X\times G(\AAA_f)/K
\]
be the associated Shimura variety. 
We refer to \cite{Milne} and references therein for definitions and facts related to Shimura varieties. Throughout the paper we always assume that $G$ is the 
generic Mumford-Tate group on $X$. This is a standard convention in the theory of Shimura varieties.
We will write a point of $\Sh_K(G,X)$ as $s=\overline{(h,t)}$ - the double class of the element $(h,t) \in X \times G(\AAA_f)$.

The \emph{Hecke orbit} of a point of $\Sh_K(G,X)$ is defined as follows.

\begin{defi}[Hecke orbit] Let~$s=\ol{(h, t)}$ be a point of~$\Sh_K(G,X)$. We define the \emph{Hecke orbit} of the point~$s$ in~$\Sh_K(G,X)$ to be the set
$$
\Hcal(s) = \left\{ \ol{(h, t\cdot g)}~\middle|~g \in G(\AAA_f) \right\}\subseteq \Sh_K(G,X).
$$
\end{defi}

We refer to \cite{Moonen} for a definition of \emph{weakly special subvarieties}\footnote{A concise characterisation of \emph{weakly special subvarieties}, the one studied in \cite{Moonen}, is the following: they are the complex algebraic subvarieties of~$\Sh_K(G,X)$ which are the  images of totally geodesic submanifolds of the symmetric space~$X\times\{t\}$, with some~$t\in G(\AAA_f)$. Each such subvariety is isomorphic to a component of some Shimura variety, that is, an arithmetic quotient of a Hermitian symmetric domain. Weakly special subvarieties containing smooth special points are called special subvarieties. A special subvariety is a 
component of the image of a morphism of Shimura varieties induced by a morphism of Shimura data.
In a certain sense a weakly special subvariety is a translate of a special subvariety.} of $\Sh_K(G,X)$ and related notions. 
Note that several equivalent definitions (different in flavour) are given in \cite{U1}, \cite{UY3} and \cite{UY2}.

The Andr\'e-Pink-Zannier conjecture is the following statement.

\begin{conj}[Andr\'e-Pink-Zannier] \label{APZ}
Irreducible components of the Zariski closure of any subset~$\Sigma$ of $\Hcal(s)$ are weakly special subvarieties.
\end{conj}

This conjecture was formulated by Andr\'e for curves in Shimura varieties in \cite{A}, then by Pink 
for arbitrary subvarieties of mixed Shimura varieties in \cite{Pinky} and independently
by Zannier (unpublished). We therefore refer to this conjecture as the
Andr\'e-Pink-Zannier conjecture.

For \emph{curves} in Shimura varieties of abelian type, conjecture \ref{APZ} was proved by Orr in \cite{Orr}.
Orr also proves in \cite{Orr} that conjecture \ref{APZ} is a special case of the Zilber-Pink conjecture on unlikely intersections in Shimura varieties. 

In this paper we deal with the `$S$-Hecke orbit' version of Conjecture~\ref{APZ}. We actually prove much 
stronger statements than the conclusion of \ref{APZ}.

\subsection{The $S$-Andr\'{e}-Pink-Zannier conjecture}
We consider the following weaker notion of a Hecke orbit.

\begin{defi}[$S$-Hecke orbits]\label{DefiS-Hecke}
 Let $S$ be a finite set of prime numbers. 
\begin{enumerate}
\item
Write~$(g_\ell)_\ell$ an element of~$G(\AAA_f)$ viewed as a restricted product indexed by primes~$\ell$. We denote~$G_S$ the subgroup of~$G(\AAA_f)$  consisiting of elements~$(g_\ell)_\ell$ such that $g_\ell = 1$ for~$\ell\notin S$.\\ As~$S$ is a finite set, we may identify~$G_S$ with~$\prod_{\ell\in S} G(\QQ_\ell)$, or equivalently with~$G(\QQ_S)$ where~$\QQ_S=\prod_{\ell\in S} \QQ_\ell$.
\item 
For a point~$s=\ol{(h, t)}$ of $\Sh_K(G,X)$, we define the \emph{$S$-Hecke orbit} of~$s$ to be the subset
$$
\Hcal_S(s) = \left\{ \ol{(h, t\cdot g)}~\middle|~ g \in G_S \right\}\subseteq \Hcal(s).
$$
\end{enumerate}
\end{defi}

Let $E$ be a field of definition of $s$. The theory of Shimura varieties (the fact that Hecke correspondences are defined over the reflex field $E(G,X)$) shows that
points of $\Hcal(s)$ and $\Hcal_S(s)$ are defined over $\overline{E}$.

We introduce the following definition which will play a central role in what follows.

\begin{defi}[$S$-Shafarevich property] \label{SSha}
The point $s$ is said to satisfy the $S$-Shafarevich property or to be
of $S$-Shafarevich type if for every finite extension $F \subset \overline{E}$
of $E$, $\Hcal_S(s)$ contains only finitely many points defined over $F$.
\end{defi}

We immediately observe that the $S$-Shafarevich property is invariant by 
finite morphisms induced by the inclusions of compact open subgroups
$K' \subset K$.
The conclusions of Conjectures \ref{APZ} and \ref{APS} as well as those 
of our main theorems \ref{Theoreme1} and \ref{Theoreme2} are also invariant by 
replacing $K$ by a subgroup of finite index.
Therefore, throughout the paper we always assume \emph{the group $K$ to be 
\underline{neat}}.
More precisely, we choose the group $K$ as follows.
\begin{enumerate}
\item $K$ is a product $K = \prod K_p$ of compact open subgroups 
$K_p \subset G(\QQ_l)$.
\item Fix $l \geq 3$ and $l \notin S$. We assume that 
$K_l$ is contained in the group of elements congruent to the identity modulo $l$ 
with respect to some faithful representation $G \subset \GL_n$.
\end{enumerate}

As noted just above the assumptions cause no loss of any generality but
avoid many annoying technicalities. These assumptions will be kept throught the paper.

We will examine this property in great detail in section \ref{SectionGalois}. In particular,
we will obtain a group theoretic characterisation of this property.

We now formulate the corresponding weaker form of Conjecture~\ref{APZ} which 
will be the main object of study in this paper.

\begin{conj}[Andr\'e-Pink-Zannier conjecture for $S$-Hecke orbits] \label{APS}
Irreducible components of the Zariski closure of any subset of $\Hcal_S(s)$ are weakly special subvarieties.
\end{conj}
In this paper we prove results in the direction of  Conjecture~\ref{APS}. There are some known cases of this conjecture, all are implied by our results. In particular,
\begin{itemize}
 \item that Conjecture~\ref{APS} holds whenever~$(G,X)$ is of abelian type (proved in \cite{OrrThesis} by a different method)
 \item that Conjecture~\ref{APS} holds if~$s\in \Sh_K(G,X)$ is a special point (this is a special case of the 'weak Andr\'e-Oort conjecture' proven in \cite{KY});
\end{itemize}
 
The conclusion of our main theorem \ref{Theoreme1} holds under the assumption that the $S$-Shafarevich property (\ref{SSha}) holds.
We actually prove equidistribution results that are much stronger than conclusions of Conjecture \ref{APS}.

\subsection{Galois monodromy properties}\label{introgalois}

Let $E \subset \CC$ be a field 
over which the point $s = \overline{(h,t)}$ is defined. 
Let $M \subset G$ denote the Mum\-ford-Tate group of~$h$.
We assume that $E$ contains the reflex field of the Shimura datum $(M,M(\RR) \cdot h)$.

\begin{defi}\label{defiintro}
\begin{enumerate}
\item \label{defimonorep} 
There exists a continuous ``$S$-adic monodromy representation'' 
$$
{\rho_{h,S}} \colon \Gal(\ol E / E) \lto M(\AAA_f) \cap K \cap G_S.
$$
which has the  property that for any $g \in G_S$ and $\sigma \in \Gal(\ol E / E)$,
\begin{equation}\label{Galois carac}
\sigma\left(\ol{(h,tg)}\right) = \ol{(h, \rho(\sigma) \cdot tg)}.
\end{equation}
\item \label{defimonogroup}
The~\emph{$S$-adic monodromy group}, which we will denote by~$U_S$, is defined as the image of~$\rho_{h,S}$:
\[
U_S=\rho_{h,S}\left(\Gal(\ol E / E)\right).
\]
This is a compact $S$-adic Lie subgroup of~$M(\AAA_f)\cap G_S\simeq M(\QQ_S)$.
\item The algebraic~\emph{$S$-adic monodromy group}, denoted~$H_S$, is defined as the algebraic envelope in~$M_{\QQ_S}$ of the subgroup~$U_S$ of~$M(\QQ_S)$:
\[
H_S=\ol{U_S}^{\text{Zar.}}
\]
This is an algebraic group over~$\QQ_S$. 
We write~$H^0_S$ for its neutral component.
\end{enumerate}
\end{defi}
\begin{defi}\label{ProprietesGaloisiennes} Let~$s$ be a point of~$\Sh_K(G,X)$ and~$E$ a field of definition of~$s$ isuch that the associated $S$-adic monodromy representation~\[\rho:\Gal\left(\ol{E}/E\right)\to M(\AAA_f)\cap G_S\] is defined. 
We say that the point~$s$ over~$E$ is
\begin{enumerate}
\item \label{S-MT}
	 of $S$-\emph{Mumford-Tate type} if~$U_S$ is open in $M(\AAA_f) \cap G_S$;
\item \label{S-Sha}
	 of \emph{$S$-Shafarevich type} if for every finite extension~$F$ of $E$, there are only finitely many $F$-rational points in~$\Hcal_S(s)$ (i.e. property in definition \ref{SSha} holds);
\item \label{S-Tate}
	 of \emph{$S$-Tate type} if~$M$ and~$H^0_S$ share the same centraliser in~$G_S$:
$$
Z_{G_S}(H_S^0) = Z_{G_S}(M).
$$
\end{enumerate}
We say that the point~$s$ over~$E$
\begin{enumerate}\setcounter{enumi}{3}
\item \label{S-semisimple} satisfies $S$-\emph{semisimplicity} if $H_S$ is a reductive group;
\item \label{Algebraicity} satisfies~$S$-\emph{algebraicity} if the subgroup~$U_S$ of~$H_S(\QQ_S)$ is open. 
\end{enumerate}
This latter is equivalent to the Lie algebra of~$U_S$ being algebraic (in the sense of Chevalley, cf. \cite[\S7]{BorelLAG}), as a~$\QQ_S$-Lie subalgebra of the Lie algebra of~$M_{\QQ_S}$.
\end{defi}

We will examine the interrelations between these properties in detail 
 in section \ref{SectionGalois}. We will also provide examples where these
various assumptions hold and where they do not.
In \ref{SectionGalois}, we will make apparent that these properties are very heavily dependent on the choice of field $E$, especially its property of being of 
finite type over $\QQ$ or not.

Note that a point of~$S$-Mumford-Tate type is obviously of~$S$-Tate type. These properties hold 
notably for a special point, by definition of a canonical model.
We will prove that $S$-Tate type property~\eqref{S-Tate} implies the algebraicity of~$U_S$ if~$E$ is of finite type. As a consequence, for such~$E$, the~$S$-Mumford-Tate type property~\eqref{S-MT} is equivalent to its \emph{a priori} weaker variant:~$H_S^0=M_{\QQ_S}$.

Our main results are proved under the $S$-Shafarevich hypothesis~\eqref{S-Sha} which as we will see in section \ref{SectionGalois} is implied by
(and in fact not far from being equivalent to) the $S$-Tate hypothesis~\eqref{S-Tate} together with the $S$-semisimplicity assumption~\eqref{S-semisimple}. 
More precisely, the main result of section \ref{SectionGalois} is that the $S$-Shafarevich
property is equivalent to $S$-semisimplicity together with the property that
the centraliser of $U_S$ in $G_S$ is compact modulo the centraliser of $M$ in $G_S$.

\subsection{Notions related to equidistribution.}

Let $s \in \Sh_K(G,X)(E)$. Recall that points of $\Hcal(s)$ are defined over $\ol E$.

To any $z \in \Hcal(s)$ we attach a probability measure with finite support that we define as follows. In what follows, for a point $z$ in $\Sh_K(G,X)$, we refer to the following set as its  Galois orbit:
$$
\Gal\left(\ol E/E\right) \cdot z = \left\{ \sigma(z)\,\middle|\,\sigma \in \Gal\left(\ol E/E\right) \right\}.
$$

\begin{defi} \label{measure}
Let $z = \ol{(x, t)}$ be a point of $\Hcal_S(s)$. We define
\begin{equation}\label{defi mesure}
\mu_z =  \frac{1}{\#\Gal\left(\ol E/E\right) \cdot z}\cdot\sum_{\zeta \in   \Gal\left(\ol E/E\right)\cdot z} \delta_{\zeta}
\end{equation}
where $\delta_{\zeta}$ is the Dirac mass at $\zeta$.
\end{defi}

We introduce a class of groups, arising from the use of Ratner's theorem.

\begin{defi}\label{defi S Ratner class}
A connected $\QQ$-subgroup~$L$ of $G$ is said to be of \emph{$S$-Ratner class} if its Levi subgroups are semisimple and for every~$\QQ$-quasi factor~$F$ of this Levi,~$F(\RR\times \QQ_S)$ is not compact\footnote{Equivalently such factor~$F$ can not be anisotropic simultaneously over~$\RR$ and every of the~$\QQ_v$ for~$v$ in~$S$.}. 

We denote by~$L^\dagger$ the subgroup of~$L(\RR\times\QQ_S)$ generated by the unipotent  elements. Then~$L$ is of $S$-Ratner class if and only if no proper subgroup of~$L$ defined over~$\QQ$ contains~$L^\dagger$.

We denote by~$L(\RR)^{+}$ the neutral component of~$L(\RR)$ with respect to the Archimedean topology. 
\end{defi}

This class is slightly more general\footnote{Actually every Zariski connected subgroup~$L$ with semisimple Levi subgroups will be of~$S$-Ratner class for some~$S$ big enough.} than~\cite[D\'{e}f.~2.1]{U} and~\cite[Def.~2.4]{UY2}. The latter actually corresponds to the case where~$S$ is empty, in which case we have~$L^\dagger=L(\RR)^+$.

Let~$L$ be subgroup of~$G$ of $S$-Ratner class and~$(h,t)\in X\times G(\AAA_f)$. 
To such data we associate the following subset (actually a real analytic variety):
\[
	Z_{L,(h,t)}=\left\{\ol{(l\cdot h,t)}~\middle|~l\in L(\RR)^+\right\}\subseteq \Sh_K(G,X).
\]
We consider the corresponding generalisation of a notion of weakly special subvariety from~\cite{UY2}.

\begin{defi}\label{ws subm}
A \emph{weakly $S$-special real submanifold (or subvariety)}~$Z$ of~$\Sh_K(G,X)$ is a subset of the form~$Z_{L,(h,t)}$ for some subgroup~$L$ of~$G$ of~$S$-Ratner class and~$(h,t)$ in~$X\times G(\AAA_f)$.
\end{defi}
Note that the parametrising map~$L(\RR)^+\xrightarrow{l\mapsto \ol{(l\cdot h,t)}} Z_{L,(h,t)}$ induces a homeomorphism
\[
	\left.
	\left(\Gamma_{tK}\cap L(\RR)^+\right)
	\middle\backslash 
	L(\RR)^+
	\middle/
	(L(\RR)^+\cap K_h)
	\right. 
	\to Z_{L,(h,t)}
\] 
where~$\Gamma_{tK}$ is an arithmetic subgroup of~$G(\QQ)$ depending only on~$tK$
whose  intersection with~$L(\RR)^+$ is an arithmetic subgroup, necessarily a lattice. 
Indeed, there is a canonical right $L(\RR)^+$-invariant probability measure on~$\left(\Gamma_{tK}\cap L(\RR)^+\right)\backslash L(\RR)^+$. We denote~$\mu_{L,(h,t)}$
its direct image in~$Z_{L,(h,t)}$, viewed as a Borel probability measure on~$\Sh_K(G,X)$.
\begin{defi}\label{canonical} Let~$Z$ be a weakly $S$-special real submanifold~$Z$ inside of~$\Sh_K(G,X)$. The \emph{canonical probability~$\mu_Z$ with support~$Z$}, is a measure of the form~$\mu_{L,(h,t)}$ with support~$Z=Z_{L,(h,t)}$.
\end{defi}

The next lemma shows that $\mu_{L,(h,t)}$ is independent of choices.

\begin{lem}
The canonical probability measure~$\mu_Z$ is well defined: it depends only on $Z$ and not on the choice of~$(L,(h,t))$. 
\end{lem}
\begin{proof}
Firstly note that~$Z$ is almost everywhere locally isomorphic to its 
inverse image~$\wt{Z}$ in~$X\times G(\AAA_f)/K$, and~$\mu_Z$ is determined by the corresponding locally finite measure~$\mu_{\wt{Z}}$ on~$\wt{Z}$.
It will suffice to show that~$\mu_{\wt{Z}}$ is intrinsic up to a locally constant scaling factor, the latter being characterised by~$\mu_Z$ being a probability and~$Z$ being connected. Endow~$X$ with a~$G(\RR)$-invariant Riemannian structure, which we extend to~$X\times G(\AAA_f)/K$. Then~$\mu_{\wt{Z}}$ is locally proportional to the volume form of the induced Riemannian structure on~$\wt{Z}$. It suffices to check it for an orbit~$L(\RR)^+\cdot x$ in~$X$.
But the~$L(\RR)^+$-invariant measure on~$L(\RR)^+\cdot x$, the Haar measure,
is unique up a factor, and, as~$L(\RR)^+\leq G(\RR)$ acts by isometries on~$X$, the
Riemannian volume form on~$L(\RR)^+\cdot x$ is a Haar measure.
\end{proof}

\subsubsection{Weakly special subvarieties} We will be involved with the notion of weakly special subvariety 
only in the following two statements. The first is a slight generalisation of~\cite[Prop.~2.6]{UY2}. This in fact is a direct consequence of the hyperbolic Ax-Lindemann-Weierstrass theorem proven in \cite{KUY}.

\begin{prop}[\cite{KUY}, \cite{UY}]\label{Ax}
 The Zariski closure of a weakly $S$-special real submanifold 
is a weakly special subvariety.
\end{prop}
The following is a generalisation of an observation\footnote{We quote:~``In the case where~$h$ viewed as a morphism from~$\SSS$ to~$G_\RR$ factors through~$H_{\RR}$, the corresponding real weakly special subvariety has Hermitian structure and in fact is a weakly special subvariety in the usual sense''.} of~\cite[p.~2]{UY2}.

\begin{prop}Let~$Z=Z_{L,(h,t)}$ be a weakly $S$-special real submanifold of~$\Sh_K(G,X)$.
If~$L$ is normalised by~$h$, then~$Z$ a weakly special subvariety: it is Zariski closed.
\end{prop}
\subsection{Main theorems} 
We may finally state our main theorems, which give a stronger form of Conjecture~\ref{APS}, at the cost of~$S$-Tate type assumption.

We now state our first main result.

\begin{teo}[Inner Equidistributional $S$-Andr\'{e}-Pink-Zannier] \label{Theoreme1}
Let~$s$ be a point of ~$\Sh_K(G,X)$ defined over a field $E$ such that 
\item the point $s$ is of  $S$-Shafarevich type.

Let~$(s_n)_{n\geq0}$ be a sequence of points in the~$S$-Hecke orbit~$\Hcal_S(s)$ of~$s$, and denote $(\mu_n)_{n\geq0}$ the sequence 
of measures attached to the~$s_n$ as in definition \ref{measure}.

After possibly extracting a subsequence and replacing $E$ by a finite extension, 
there exists a finite set~$\cF$ of weakly $S$-special real submanifold $Z$ with canonical probability measure~$\mu_Z$  (as defined in Definitions~\ref{ws subm} and~\ref{canonical})  such that 
\begin{enumerate}
 \item  \label{TheoremeLimite}	 the sequence~$(\mu_n)_{n\geq 0}$ tightly converges to~$\mu_\infty=\frac{1}{\#\cF}\sum_{Z\in \cF}\mu_Z$,
 \item and for all $n\geq 0$, we have~${\rm Supp}(\mu_n)\subseteq\Supp(\mu_\infty)=\bigcup_{Z\in\cF}  Z$. \label{TheoremeInner}	
\end{enumerate}

\begin{enumerate}\setcounter{enumi}{2}
 \item If $s$ is of $S$-Mumford-Tate type,
 then every $Z$ in~$\cF$ is a weakly special subvariety.\label{TheoremeMT}
\end{enumerate}
\end{teo}
\noindent We may refer to property~\eqref{TheoremeLimite} as ``equidistribution'' property -- a shorthand for a convergence of measures --
and to property~\eqref{TheoremeInner} by saying this equidistribution is ``inner''. 

Note that conclusion~\eqref{TheoremeMT} applies to special points, in which case every~$Z$ in~$\cF$ is actually a \emph{special} subvariety.

We deduce from theorem  \ref{Theoreme1} the following theorem, which is more directly 
related to the Andr\'{e}-Pink-Zannier conjecture. Let us stress that, in the deduction process, 
we need not only~\eqref{TheoremeLimite}, but also~\eqref{TheoremeInner}, from 
Theorem~\ref{Theoreme1}.

\begin{teo}[Topological and Zariski $S$-Andr\'{e}-Pink-Zannier] \label{Theoreme2}
Let~$s$ be a point of a Shimura variety~$\Sh_K(G,X)$ defined over a field~$E\subseteq \CC$.
Consider a subset $\Sigma\subseteq \Hcal_S(s)$ of its~$S$-Hecke orbit and denote
$$
\Sigma_E = \Gal\left(\ol E/E\right)\cdot \Sigma 
=
\left\{ \sigma(x)~\middle|~\sigma \in \Gal(\ol E/E), x \in \Sigma \right\}.
$$
Then
\begin{enumerate}
\item \label{Theoreme2-1}
 If $s$ is of $S$-Shafarevich type then the topological closure of~$\Sigma_E$
 is a finite union of weakly $S$-special real submanifolds;

 Furthermore, the Zariski closure of $\Sigma$ is a finite union of weakly special subvarieties.
 
\item \label{Theoreme2-2}
If $s$ is of $S$-Mumford-Tate type, then the  topological closure of~$\Sigma_E$ is a finite union of weakly special subvarieties;

\end{enumerate}
\end{teo}

We will prove that~\eqref{Theoreme2-1} holds whenever
~\eqref{TheoremeLimite} and~\eqref{TheoremeInner} from Theorem~\ref{Theoreme1} hold for sequences in~$\Sigma$. The second statement will then be the 
the consequence of Proposition \ref{Ax}.

When the $S$-Mumford-Tate property holds, the conclusion ~\eqref{TheoremeMT} from Theorem~\ref{Theoreme1} will imply
~\eqref{Theoreme2-2}. 

\subsubsection{A converse statement.}

Let us end with a statement emphasizing the importance of property~\eqref{TheoremeInner} of Theorem~\ref{Theoreme1}. This statement 
makes precise the idea that property~\eqref{TheoremeInner} of Theorem~\ref{Theoreme1} implies the~$S$-Sha\-fa\-rev\-ich property.

This shows that the $S$-Shafarevich property assumption is essential and optimal.

\begin{prop} Let~$s$ be a point in a Shimura variety~$\Sh_K(G,X)$ defined over a field~$E$ and let~$\Hcal_S(s)$ be its~$S$-Hecke orbit.

Assume that for any sequence~$(s_n)_{n\geq0}$ in~$\Hcal_S(s)$, for any 
finite extension~$F$ of~$E$, there is an extracted 
subsequence for which the associated measure~$\mu_n$ converges weakly
to a limit~$\mu_\infty$ in such a way that 
\begin{equation}\label{innerprop}
\forall n\geq 0, \Supp(\mu_n)\subseteq \Supp(\mu_\infty).
\end{equation}

Then~$s$ is of~$S$-Shafarevich type.
\end{prop}
\begin{proof} Assume for contradiction that~$s$ is not of~$S$-Shafarevich type.
Then there is a finite extension~$F$ of~$E$ such that there is an infinite
sequence~$(s_n)_{n\geq 0}$ of pairwise distinct~$F$-rational points in~$\Hcal_S(s)$. After possibly extracting a subsequence, we may assume that this 
sequence is convergent or is divergent in~$\Sh_K(G,X)$.

As these~$s_n$ are rational points, the associated measures~$\mu_n$ are Dirac masses. We recall that weak convergence of Dirac masses is induced by
convergence in the Alexandroff compactification, with the point
at infinity corresponding to the zero measure.

If~$(s_n)_{n\geq 0}$ is divergent, so is any subsequence,
and the measure~$\mu_\infty$ will be the~$0$ measure, in which case~\eqref{innerprop} may not hold.

If~$(s_n)_{n\geq 0}$ converges to~$s_\infty$, then~$\mu_\infty$ will be the Dirac measure~$\delta_{s_\infty}$, and~\eqref{innerprop} means that~$(s_n)_{n\geq 0}$ is a stationary sequence, which it cannot be since the~$s_n$ are pairwise distinct.
This yields a contradiction.
\end{proof}

\subsection{Plan of the Article} In Section~\ref{consequence} we explain how to deduce Theorem~\ref{Theoreme2}
from Theorem~\ref{Theoreme1}. 
Section~\ref{SectionGalois} reviews Galois representations,
various properties listed before ($S$-Mumford-Tate, $S$-Shafarevich, $S$-Tate, etc)
and relations between them. We in particular prove useful and practical group-theoretic
characterisation of the $S$-Shafarevich property.
We also provide examples and counterexamples of when the properties do and do not hold depending on the field $E$. 
We believe the contents and results of this section to be of independent interest.

The sections that follow are devoted to the proof of  Theorem~\ref{Theoreme1}.
\begin{itemize}
\item 
Section~\ref{Reduction} explain how to reduce to a situation
falling under the scope of application of~\cite{RZ}. It ends 
by invoking~\cite{RZ}, which immediatly gives us~\eqref{TheoremeLimite}
of Theorem~\ref{Theoreme1}
\item Section~\ref{SectionRZ}  then discusses how
to get~\eqref{TheoremeInner} of Theorem~\ref{Theoreme1}.
\item Finally Section~\ref{SectionMT} treats the
stronger conclusion we can reach under the $S$-Mumford-Tate hypothesis. 
\end{itemize}
\newpage

\section*{Acknowledgments}

Both authors were supported by the ERC grant (Project 511343, SPGSV). They gratefully acknowledge ERC's support.

The first named author is grateful to UCL for hospitality.

\section{From Inner equidistribution to Topological closures}\label{consequence}

In this section we show how to derive Theorem~\ref{Theoreme2} from Theorem~\ref{Theoreme1}.
The main result is proposition \ref{Sorite} which implies the main theorem of this section:

\begin{teo} \label{implication}
Conclusions of Theorem~\ref{Theoreme1} imply conclusions of Theorem~\ref{Theoreme2}.
\end{teo}

First we need to develop a dimension theory of weakly $S$-special submanifolds.

\subsection{Dimension and Measure in chains of weakly $S$-special real submanifolds}\label{subsec dimension}

We prove here some standard properties about inclusions of weakly~$S$-special real submanifolds, involving dimension, that we define, and their canonical measure.

\begin{defi} Let~$Z=Z_{L,(h,t)}$ be a weakly $S$-special real sub\-ma\-ni\-fold. Then we define the dimension of~$Z$
as  the codimension of the stabiliser~$K_h\cap L(\RR)$ of~$h$ in~$L(\RR)$. This is also the dimension~$L(\RR)^+\cdot h$, or equivalently~$L(\RR)\cdot h$  in~$X$, as a semialebraic set and as a real analytic variety.
\end{defi}

\begin{lem}\label{lemme canonique}
\begin{enumerate}
\item The dimension of a weakly $S$-special real sub\-ma\-ni\-fold is well defined. If~$Z_{L_1,(h_1,t_1)}=Z_{L_2,(h_2,t_2)}$,
then~
\begin{equation}\label{defi dim}
\dim\left( L_1(\RR)^+\cdot h_1\right)=\dim\left( L_2(\RR)^+\cdot h_2\right)
\end{equation}
\item Let~$Z_1\subsetneq Z_2$ be two $S$-special real submanifolds. 
Then
\begin{equation}
\dim Z_1< \dim Z_2
\end{equation}
\item  \label{mesure ortho}
Let~$Z_1$ and~$Z_2$ be two weakly $S$-special real submanifolds, such that~$\mu_{Z_1}(Z_2)\neq 0$.
Then~$Z_1\subseteq Z_2$, and~$\mu_{Z_1}(Z_2)= 1$.
\end{enumerate}
\end{lem}\label{lemme chains}

One immediately deduces the following.

\begin{cor}  \label{length}
Let~$Z_1\subsetneq \ldots \subsetneq Z_l$ be a chain of strictly included weakly $S$-special real submanifolds.
Then its length~$l$ satisfies~$l\leq 1+\dim(G)$. 
\end{cor}

From this we deduce the following.

\begin{cor}\label{coro max}
Any non empty collection~$\mathcal{F}$ of weakly $S$-special real submanifolds, partially ordered by inclusion, has maximal elements, and any element of~$\mathcal{F}$ is contained in a maximal element of~$\mathcal{F}$.
\end{cor}
\begin{proof}[Proof of the Corollary~\eqref{coro max}] By induction we may extend every chain in~$\mathcal{F}$ to a maximal one. By \ref{length}, this induction terminates after at most~$1+\dim(G)$ steps. 

The last element of a non empty maximal chain is a maximal element.
Hence any element~$f$, seen as a chain of length one, is part of  a non empty maximal chain. The last element of the latter contains~$f$ and is maximal in~$\mathcal{F}$.
If~$\mathcal{F}$ has an element, this implies that there is a maximal element.
\end{proof}

\begin{proof} Write~$Z_i$ for ~$Z_{L_i,(h_i,t_i)}$ (for~$i=1$ or~$i=2$).

\emph{We assume that the intersection~$Z_1\cap Z_2$ is not empty.}

These two subsets~$Z_1$ and~$Z_2$ of~$\Sh_K(G,X)$ are connected, and hence belong to the same connected component of~$\Sh_K(G,X)$.
This implies, as subsets in~$G(\AAA)$,
\[G(\QQ)\cdot\left( G(\RR)\times t_1K\right)=G(\QQ)\cdot \left(G(\RR)\times t_2K\right).\]
Left translating~$t_1$ with~$\gamma\in G(\QQ)$ and right translating with~$k\in K$ we may assume~$t_1=t_2$.
We have to substitute accordingly~$h_1$ with~$\gamma h_1$ and~$L_1$ with~$\gamma L_1 \gamma^{-1}$.
As we have
\[
\gamma L_1 \gamma^{-1}\cdot \gamma h_1=\gamma\left(L_1\cdot h_1\right)
\]
this does not change the notion of dimension of~$Z_1$.

\emph{We now assume that~$t_1=t_2$, which we will be denote simply~$t$.}

Let~$\Gamma_{tK}$ be the inverse image in~$G(\QQ)$ of~$tKt^{-1}$ with respect to the map~$G(\QQ)\to G(\AAA_f)$. This is the
arithmetic subgroup such that the previous component of~$\Sh_K(G,X)$ belongs to those of~$\Gamma_{tK}\backslash X\times\{t\}$.
We will identify~$X\times\{t\}$ with~$X$ for simplicity. The inverse images of~$Z_1$ and~$Z_2$ in~$X$ are
\(\wt{Z_1}=\Gamma_t\cdot L_1(\RR)^+\cdot h_1\) and \(\wt{Z_2}=\Gamma_t\cdot L_2(\RR)^+\cdot h_2\) respectively.

We may write 
\[
\wt{Z_1}\cap\wt{Z_2}=\Gamma_{tK}\cdot\left(\left( L_1(\RR)^+\cdot h_1\right)\cap \wt{Z_2}\right)
\]
and 
\begin{equation}\label{inter union}
\left(L_1(\RR)^+\cdot h_1\right)\cap \wt{Z_2}=\bigcup_{\gamma\in\Gamma_t} \left(L_1(\RR)^+\cdot h_1\right)\cap\left(\gamma_t\cdot L_2(\RR)^+\cdot h_2\right).
\end{equation}

\emph{Assume first that~$\mu_{Z_1}(Z_2)\neq 0$.}

By our definition of~$\mu_{Z_1}$, the~$\left(\Gamma_{tK}\cap L_1(\RR)^+\right)$-saturated set
\begin{equation}\label{sature}
\left(\Gamma_{tK}\cap L_1(\RR)^+\right)\cdot\left(L_1(\RR)^+\cdot h_1\right)\cap \wt{Z_2}
\end{equation}
is non negligible (Cf. Lemma~\ref{lemme neglect} proven below) in~$L_1(\RR)^+\cdot h_1$ with respect to a Haar measure on 
the homogeneous~$L_1(\RR)^+$-set~$L_1(\RR)^+\cdot h_1$. But this~\eqref{sature} is again 
the countable union~\eqref{inter union}. So there is a~$\gamma$ in~$\Gamma_{tK}$ such
that
\[
\left(L_1(\RR)^+\cdot h_1\right)\cap\left(\gamma_t\cdot L_2(\RR)^+\cdot h_2\right)
\]
is not negligible. This is a real semi-algebraic subset of the real semi-algebraic set~$L_1(\RR)^+\cdot h_1$.
We use a cylindrical cellular decomposition of tis subset. Subset of codimension~$1$ are negligible\footnote{}.
So at least one cell has codimension~$0$. It must have non empty interior. 

The orbit map~$L_1(\RR)^+\to L_1(\RR)^+\cdot h_1$ are open maps. So there is an open subset~$U$
in~$L_1(\RR)^+$ such that~$U\cdot h_1\subseteq \gamma L_2(\RR)^+\cdot h_2$. But~$L_1$ is Zariski
connected, and~$U$ is Zariski dense in~$L_1$. Hence~
\[L_1(\RR)^+\cdot h_1\subseteq (L_1\cdot h_1)(\RR)\subseteq (\gamma L_2\cdot h_2)(\RR).\]
We note that~$L_1(\RR)^+\cdot h_1$ is a connected component of~$(L_1\cdot h_1)(\RR)$. 
Likewise~$\gamma L_2(\RR)^+\cdot h_2$ is a connected component of~$(\gamma L_2\cdot h_2)(\RR)$.
But the connected~$L_1(\RR)^+\cdot h_1$ intersects the component~$\gamma L_2(\RR)^+\cdot h_2$, hence
is contained in it. It follows~$\wt{Z_1}\subseteq \wt{Z_2}$ and finally~$Z_1\subseteq Z_2$.
We have proved the last point of the lemma. 

We now turn to the first point.

\emph{We now assume~$Z_1\subseteq Z_2$ instead of~$\mu_{Z_1}(Z_2)\neq 0$.}

Then certainly~$\mu_{Z_1}(Z_2)\neq0$. We  may, and will, keep the notations above. We have already proved
\[
L_1(\RR)^+\cdot h_1\subseteq \gamma L_2(\RR)^+\cdot h_2
\]
for some~$\gamma$ in~$\Gamma_t$. It follows
\[
	\dim(L_1(\RR)^+\cdot h_1)\leq\dim( \gamma L_2(\RR)^+\cdot h_2)=\dim(L_2(\RR)^+\cdot h_2).
\]
(the equality on the right is easily checked.) If~$Z_1=Z_2$ we may echange te roles to get a converse comparison, yielding~\eqref{defi dim}: the dimension of~$Z_i$ is well defined.
This was the first point of the lemma. 

It remains to prove the second point.

\emph{We now assume~$Z_1\subsetneq Z_2$.}

Again we keep our notations. We have proved
\[
L_1(\RR)^+\cdot h_1\subseteq \gamma L_2(\RR)^+\gamma^{-1}\cdot\gamma h_2.
\]
The reverse inclusion does not hold, as it would, easily, imply~$Z_2\subseteq Z_1$. 
We may substitute our base point~$\gamma h_2$ with~$h_1$, as it does belong to the same~$\gamma L_2(\RR)^+\gamma^{-1}$ orbit. We deduce
\[
L_1(\RR)^+\cdot h_1\subsetneq \gamma L_2(\RR)^+\gamma^{-1}\cdot h_1.
\]
Assume by contradiction that both sides have same dimension. The orbit~$L_1(\RR)^+\cdot h_1$ is closed
in~$X$, and \emph{a fortiori} closed in~$L_2(\RR)^+\gamma^{-1}\cdot h_1$. Furthermore~$L_1(\RR)^+\cdot h_1\to L_2(\RR)^+\cdot h_1$, as a map of differential manifolds, is a submersion at at least a point, by equality of dimensions.
It is a submersion everywhere y homogeneity, hence an open map. As a consequence,~$L_1(\RR)^+\cdot h_1$ is not only closed, but open as well in~$L_2(\RR)^+\cdot h_1$. As~$L_2(\RR)^+\gamma^{-1}\cdot h_1$ is connected, we deduce
\[
L_1(\RR)^+\cdot h_1= \gamma L_2(\RR)^+\gamma^{-1}\cdot h_1.
\]
That is a condradiction. This ends our proof.
\end{proof}

We finish this section by proving a lemma used in the proof above.

\begin{lem}\label{lemme neglect} Endow on~$(\Gamma\cap L )\backslash L$ with a Haar measure. Then the inverse of a negligible 
subset in~$(\Gamma\cap L )\backslash L$ is negligible in~$L$, with respect to a Haar measure.
\end{lem}
\begin{proof} Let~$N$ be a negligible subset in~$(\Gamma\cap L )\backslash L$, and~$\wt{N}$ its inverse
image in~$L$. The lemma amounts to proving that
\begin{equation}\label{neglect}
\int_{l\in L} 1_{\wt{N}}~dl=0
\end{equation}
where~$1_{\wt{N}}$ is the characteristic function  of~$\wt{N}$ and~$dl$ is a Haar measure on~$L$.
Let~$K$ be a compact subset in~$L$. And consider, as a real function~$(\Gamma\cap L )\backslash L\to\RR$,
\[
f:(\Gamma\cap L)\cdot l\mapsto \int_{\gamma \in \Gamma\cap L} 1_{K\cap N}(\gamma l)~d\gamma,
\]
where~$d\gamma$ is a Haar measure on~$(\Gamma\cap L )$.
Then we have, see~\cite[VII~\S2.1]{BBK-I-VII},
\[
\int_{l\in L} 1_{\wt{N}\cap K}~dl=\int_{(\Gamma\cap L )\backslash L} f(x) dx
\]
where~$dx$ is the quotient Haar measure (cf. loc. cit.) on~$(\Gamma\cap L )\backslash L$. But the
support of~$f$ is contained in~$N$, hence is negligible. The last integral eveluates as zero.
By choosing increasing compact subsets whose union is~$L$, we, by the monotone convergence~$\lim_K 1_{\wt{N}\cap K}=1_{\wt{N}}$,
deduce~\eqref{neglect}.
\end{proof}

\subsection{Topological and Zariski closures} \label{topologicalsub}
We place ourselves in the situation of Theorem~\ref{Theoreme2}. In particular we assume that~$s$ is of~$S$-Shafarevich type.

Let $\Sigma$ be as in theorem \ref{Theoreme2}. It is a countable set, and we write it as
~$\Sigma = \{ s_n , n \geq 0 \}$. Let~$(\mu_n)_{n\geq0}$ be the sequence of probability measures attached to~$(s_n)_{n\geq0}$ as in Definition~\ref{measure}. As the~$S$-Shafarevich hypothesis is assumed, we are permitted to
invoke Theorem~\ref{Theoreme1} for any infinite subsequence.

 Our proof of Theorem~\ref{Theoreme2} relies on the following.
 
\begin{prop}\label{Sorite} We consider the following situation.
\begin{itemize}
\item[---] Let~$\mathcal{S}$ be the set of supports of limits of converging subsequences of~$(\mu_n)_{n\geq0}$.
\item[---] Let~$\mathcal{Z}$ be the collection of $S$-special real submanifolds~$Z$ such that the canonical measure~$\mu_Z$ occurs in the composition of the limit of a converging subsequence of~$(\mu_n)_{n\geq0}$. 
\item[---] We endow~$\mathcal{Z}$ with the partial order induced by inclusion. Let~$\mathcal{M}$ be the subset of maximal elements in~$\mathcal{Z}$.
\end{itemize}
We have the following.
\begin{enumerate}[label=\roman*)]
\item \label{Scholie1} Every support~$S$ belonging to~$\mathcal{S}$ is a finite union of finitely many weakly $S$-special real submanifolds belonging to~$\mathcal{Z}$. If~$s$ is of $S$-Mumford-Tate type, then the~$Z$ belonging to~$\mathcal{Z}$ are actually weakly special subvarieties.
\item \label{Scholie2} Every element~$Z$ of~$\mathcal{Z}$ is included in a maximal element of~$\mathcal{Z}$, an element belonging to~$\mathcal{M}$.
\item \label{Scholie3} The subset~$\mathcal{M}$ of maximal elements of~$\mathcal{Z}$ is a finite subset.
\item \label{Scholie4} Every~$Z$ in~$\mathcal{Z}$, or~$S$ in~$\mathcal{S}$, is contained in the topological closure of~$\Sigma_E$.
\item \label{Scholie5} All but finitely many elements of~$\Sigma_E$ are in~$\bigcup_{Z\in \mathcal{M}} Z$.
\item \label{Scholie6} The topological closure of $\Sigma_E$ is a finite union 
of weakly $S$-special real manifolds. 
\item \label{Scholie7} The Zariski closure of $\Sigma_E$ is a finite union of weakly special subvarieties.
\end{enumerate}
\end{prop}

To justify the definition of $\mathcal{Z}$, we need to show
 that the~$\mu_Z$ that occur in the sum  with a nonzero coefficient of a limit measure~$\mu$ are defined unambiguously.
 
 This is a consequence of the following:
 
 \begin{lem}
Any finite set of canonical measures~$\mu_Z$ is linearly independent.
\end{lem}
\begin{proof} Consider a linear combination~$\mu=\lambda_1\mu_{Z_1}+\ldots+\lambda_n\mu_{Z_n}$. We may compute~$\mu(Z)$ by using~\ref{lemme canonique}\,\eqref{mesure ortho}. It follows that we recover the coefficient of~$\mu_Z$ as the measure~$\mu(Z)$ minus the coefficients associated with subvarieties in~$Z$. To see it is well defined, we argue by induction on the dimension of~$Z$ to check that we thus obtain only finitely many non zero coefficients, because these agree with the~$\lambda_i$. We refer to~Corollary~\ref{lemme chains} for justification why this induction is legit.

So the coefficient of~$\mu_Z$ in~$\mu$ is uniquely defined.
\end{proof}

We proved that \ref{Theoreme1} implies $(1)$ and $(2)$ of Theorem \ref{Theoreme2}.
We split the proof of ~\ref{Scholie3} into two lemmas below. Lemma~\ref{lem_sh3} is
\ref{Scholie3}.

\begin{proof}[of  Proposition~\ref{Sorite}]

The statement~\ref{Scholie1} is a direct consequence of  Theoreme \ref{Theoreme1}.

The statement \ref{Scholie2} is Corollary~\ref{coro max} from the previous section.

To prove \ref{Scholie3} we prove two lemmas. The statement \ref{Scholie3} is lemma
~\ref{lem_sh3}.

\begin{lem}
The set~$\mathcal{Z}$ is countable.
\end{lem}
\begin{proof} Every element of~$\mathcal{Z}$ can be associated with the data
\begin{itemize}
\item of some group of Ratner class, which is a algebraic subvariety over~$\QQ$ of~$G$, hence belong to a countable class;
\item of some point in the $S$-Hecke orbit~$\Sigma$, which is countable.
\end{itemize}
As there only finitely many possibilities for these data, we can construct at most countably many elements in~$\mathcal{Z}$.
\end{proof}

\begin{lem}\label{lem_sh3}
The set~$\mathcal{M}$ is finite.
\end{lem}
\begin{proof}[Proof of the last claim.] Assume for contradiction that~$\mathcal{M}$ is infinite. It is countable. Hence we can we can arrange its elements as a sequence $\mathcal{M} = (M_n)_{n\geq 0}$ such that the $M_n$ are the distinct maximal (for inclusion) elements of $\cZ$. We arrange~$\cZ$ likewise in a sequence~$(Z_n)_{n\geq0}$.

Define~$S_n=\Sh_K(G,X)\smallsetminus \bigcup_{i<n} M_i$. This is an open subset.
By maximality of the~$M_i$, we have~$\mu_{M_i}(M_j)=0$ whenever~$i\neq j$ by~Lemme~\ref{lemme canonique}\,\eqref{mesure ortho}. Hence~$S_n$ is of full measure for~$\mu_{M_n}$.

We will use a diagonal argument.

By definition, there is a convergent subsequence, say~$(\mu^{(n)}_m)_{m\geq 0}$, of the sequence~$(\mu_m)_{m\geq 0}$ such that its limit, say~$\mu^{(n)}_\infty$, admits~$\mu_{M_n}$ as a component, with some non zero coefficient~$\lambda_n$.
By convergence, there is some~$N_n$ such that for~$m\geq N_n$ we have~$\mu^{(n)}_m(S_n)\geq\lambda_n/2$.
Write~$\mu^{(n)}_{N_n}$. Consequently~$\Supp\nu_n$ is not included in~$M_1\cup\ldots \cup M_{n-1}$. We deduce that no finite union of subsets~$M$ from~$\mathcal{M}$ can support infinitely many of the~$\nu_n$. As any~$Z$ from~$\mathcal{Z}$ is contained in some~$M$ from~$\mathcal{M}$, by~\eqref{Scholie2}, no finite union of such~$Z$  can support infinitely many of the~$\nu_n$. A fortiori no~$S$ from~$\mathcal{S}$  can support infinitely many of the~$\nu_n$.

But, by Theorem~\ref{Theoreme1} we may extract a subsequence from~$(\nu_n)_{n\geq0}$ which is converging, say with limit~$\nu’_\infty$, satisfying the conclusions of Theorem~\ref{Theoreme1}, and
\[
\Supp(\nu’_n)\subseteq \Supp(\nu’_\infty)\in\mathcal{S}.
\]
This yields a contradiction.
\end{proof}

The statement~\ref{Scholie4} is obvious: the topological closure of~$\Sigma_E$ is a closed subset containing 
the support of the~$\mu_n$, and hence contains the support of any limit of a subsequence of~$(\mu_n)_{n\geq0}$.

Finally, to prove~\ref{Scholie5},
assume for contradiction that there exists an infinite subsequence $(s_n)$ of points of the set $\Sigma_E$ which are
not in $\bigcup_{Z\in \mathcal{M}}Z$. Let $(\mu_n)$ be the associated sequence of measures as defined in \ref{measure}.
 By theorem \ref{Theoreme1},
after possibly extracting a subsequence, we may assume that $(\mu_n)$ converges to a measure $\mu$ whose support contains ${\rm Supp}(\mu_n)$ for all $n$.
By definition, we have that  
$$s_n\in\Supp(\mu_n)\subseteq{\rm Supp}(\mu) \subset \bigcup_{S\in \mathcal{S}} S=\bigcup_{Z\in \mathcal{Z}} Z=\bigcup_{Z\in \mathcal{M}} Z.
$$
This contradicts the choice of $s_n$.

The statement~\ref{Scholie6} follows directly from \ref{Scholie5} and
statement \ref{Scholie7} follows from \ref{Scholie6} and the fact that 
the Zariski closure of a weakly $S$-special manifold is a weakly special subvariety (Proposition \ref{Ax}).

We have finished proving proposition~\ref{Sorite} and hence Theorem \ref{implication}.

\end{proof}

\newpage
\newcommand{\Lpp}{L^\ddagger}

\section{Review of Galois monodromy representations}\label{SectionGalois}
\subsection{Construction of representations.}

In this section we recall the construction of Galois representations attached to points in Hecke orbits, and then specialise to $S$-Hecke orbits. 

The contents of this section are mostly taken from Section 2 of \cite{UY1}. There it
was assumed that $E$ was a number field, however all arguments carry over 
verbatim to an arbitrary field of characteristic zero.

Let $\Sh(G,X)$ be the Shimura variety of infinite level, it is the profinite cover
\[
	\Sh(G,X)(\CC)=\varprojlim_K \Sh_K(G,X)
\]
with respect to the finite maps induced by the inclusions of compact open subgroups.
By Appendix to \cite{UY1}, the centre $Z$ of $G$ has the property that 
$Z(\QQ)$ is discrete in $G(\AAA_f)$.
It follows (Theorem 5.28 of \cite{Milne}), we have
$$
\Sh(G,X)(\CC)=G(\QQ)\backslash\left(X\times G(\AAA_f)\right).
$$
The scheme $\Sh(G,X)$ is endowed with the right action of $G(\AAA_f)$
which is defined over the reflex field $E(G,X)$.

We let 
$$
\pi \colon \Sh(G,X) \lto \Sh_K(G,X)
$$
be the natural projection.
The Shimura varieties $\Sh(G,X)$ and $\Sh_K(G,X)$ are defined over the reflex field $E(G,X)$ and so is the map $\pi$.

Let $s = \overline{(h,t)}$ be a point of $\Sh_K(G,X)$ defined over a field $E$. 
Lemma 2.1 of \cite{UY1} shows that the fibre $\pi^{-1}(s)$ has a 
transitive fixed point free right action of $K$. Explicitly
$$
\pi^{-1}(s) = \overline{(h,tK)}
$$
and the action of $K$ is the obvious one.
Furthermore, $\Gal(\overline{E}/E)$ acts on $\pi^{-1}(s)$ and this action commutes
with that of $K$. This is a consequence of the theory of canonical models of Shimura varieties (see \cite{Milne} and \cite{Deligne}).
By elementary group theoretic Lemma 2.4 of \cite{UY1}, we obtain a
morphism 
$$
\rho_{s} \colon \Gal(\overline{E}/E) \lto K
$$ 
such that the Galois action on $\pi^{-1}(s)$ is described as follows:
$$
\sigma((h, tk)) = (h, k \rho_s(\sigma)).
$$
This representation is continuous since an open subgroup $K'$ of $K$ is of finite index
and hence $\rho_{s}^{-1}(K')$ contains $\Gal(\overline{E}/F)$ for a finite extension
$F/E$.

The representation $\rho_{s}$ has the following fundamental property.
Let $M$ be the Mumford-Tate group of $h$. To $M$ one associates the Shimura datum
$(M,X_M)$ with $X_M = M(\RR) \cdot h$. Let $E_M$ be the reflex field
of $(M,X_M)$ and $E'$ the subfield of $\ol{E}$ generated by $E$ and $E_M$.

\begin{prop}
We have
$$
\rho_{s}(\Gal(\overline{E}/E')) \subset M(\AAA_f) \cap K.
$$
\end{prop}
\begin{proof}
This is Proposition 2.9 of \cite{UY1}. This proposition is
stated with $E$ a number field, however the proof goes through
without any changes for an arbitrary field of characteristic zero.
\end{proof}

The representation $\rho_{s}$ describes the action of $\Gal(\overline{E}/E)$ on the 
Hecke orbit $\Hcal(s)$.
Let $s' = \overline{(h, tg)}$ be a point of $\Hcal(s)$.
Consider the point  $\widetilde{s'} = (h, tg)$ of $\pi^{-1}(s)$.
Let $\sigma \in \Gal(\overline{E}/E)$.
Since the action of $G(\AAA_f)$ is defined over $E(G,X)$, we have
$$
\sigma( (h, tg))= (\sigma((h,t))\cdot g = (h, \rho_s(\sigma) t)g = (h, \rho_s(\sigma) tg)
$$
By applying $\pi$ to this relation and using the fact that $\pi$ is defined over $E(G,X)$,
we obtain:
$$
\sigma(\overline{(h, tg)}) = \overline{(h, \rho_s(\sigma) tg)}.
$$

We now describe the representation $\rho_{s,S}$ and Galois action the $S$-Hecke orbit
$\Hcal_S(s)$. 
Recall that $K$ is product of compact open subgroups $K_p$ of $G(\QQ_p)$.
Let $K_S := \prod_{p \in S} K_p$ and $K^S = \prod_{p \notin S} K_p$.
We denote by $p_S$ the projection map 
$$
p_S \colon G(\AAA_f) \lto G_S.
$$
Clearly $p_S(K) = K_S$.

\begin{defi} \label{defSGal}
We define 
$$
\rho_{s,S} := p_S \circ \rho_{s} \colon \Gal(\ol{E}/E) \lto K_S. 
$$ 
\end{defi}

Let 
$$
\Sh_{K^S}(G,X) = G(\QQ) \backslash X \times G(\AAA_f) / K^S.
$$
This is a scheme defined over $E(G,X)$ endowed with a continuous right $G_S$ action 
and a morphism $\pi^S \colon  \Sh(G,X) \lto \Sh_{K^S}(G,X)$ defined over $E(G,X)$.

The maps $\pi^S \colon  \Sh(G,X) \lto \Sh_{K^S}(G,X)$ and
$\pi_S \colon  \Sh_{K^S}(G,X) \lto \Sh_{K}(G,X)$ are defined over $E(G,X)$.
Furthermore,
$$
\pi = \pi_S \circ \pi^S.
$$

Contemplation of these properties and the properties of $\rho_{s,S}$ show the following:

\begin{teo} \label{Sadic}
The morphism 
$$
\rho_{s,S} \colon \Gal(\ol{E}/E) \lto K_S
$$ 
we have constructed above has the following properties:
\begin{enumerate}
\item
The representation $\rho_{s,S}$ is continuous.
\item
Let $s' = \overline{(h, tg)}$ (with $g \in G_S$) be a point in $\Hcal_S(s)$.
Then for any $\Gal(\ol{E}/E)$, we have 
$$
\sigma(s') = \ol{(h, \rho_{s,S}(\sigma)tg)}.
$$
\item
After replacing $E$ by $EE_M$, we have
$$
\rho_{s,S}(\Gal(\ol{E}/E)) \subset M(\AAA_f) \cap K_S = M(\QQ_S) \cap K
$$
(intersection taken inside $G(\AAA_f)$). 
\end{enumerate}
\end{teo}

\subsection{Properties of $S$-Galois representations.} 

In this section we examine in detail the properties \ref{ProprietesGaloisiennes}.

\subsubsection{General assumptions.} \label{assumptions}

For the sake of the ease of reading, we recall the general situation.
Let $S$ be a finite set of places of $\QQ$. Let $(G,X)$ be a Shimura datum,
normalised so that $G$ is the generic Mumford-Tate group on $X$
and $K$ a compact open subgroup of $G(\AAA_f)$ satisfying the following conditions:
\begin{enumerate}
\item $K$ is a product $K = \prod K_p$ of compact open subgroups 
$K_p \subset G(\QQ_l)$.
\item There eists an $l \geq 3$ and $l \notin S$ such that 
$K_l$ is contained in the group of elements congruent to the identity modulo $l$ 
with respect to some faithful representation $G \subset \GL_n$.
In particular, this implies that $K$ is neat.
\end{enumerate}

We let $s = \overline{(h,t)}$ a point of $\Sh_K(G,X)$ defined over a field 
$E$. 
We let $M$ be the Mumford-Tate group of $h$. 
In what follows, for ease of notation, we will write $M$ for $M_{\QQ_S}$
or $M(\QQ_S)$ when it is clear from the contex what is meant.

As we have seen earlier in this section, there exists a continuous ``$S$-adic monodromy representation''
(we have if necessary replaced $E$ with $EE_M$): 
$$
{\rho_{h,S}} \colon \Gal(\ol E / E) \lto M(\AAA_f) \cap K_S =M(\QQ_S)\cap K_S
$$
which has the  property that for any $t \in G_S$ and $\sigma \in \Gal(\ol E / E)$,
\begin{equation}\label{Galois carac}
\sigma\left(\ol{(h,tg)}\right) = \ol{(h, \rho(\sigma) \cdot tg)}.
\end{equation}

We let $U_S \subset M(\QQ_S)$ be the image of $\rho_{(h,S)}$
and $H_S$ the algebraic monodromy group i.e. the Zariski closure 
of $U_S$.
It is immediate that properties \ref{ProprietesGaloisiennes} are invariant by replacing $K$ by an open subgroup
and $E$ by a finite extension (or equivalently $U_S$ by an open subgroup).
After replacing $E$ by a finite extension, we assume that $U_S$ (and hence $H_S$)
are connected.

\subsubsection{The algebraicity property.}

In this section we consider the $S$-algebraicity property. All results of this section
are under the assumption that $E$ is of finite type over $\QQ$.

\subsubsection*{The centre of the $S$-adic monodromy.}

The first result is that when $E$ is of finite type, the $S$-Mumford-Tate property
(and hence the $S$-algebraicity) hold for abelian $S$-adic representations.

\begin{prop}\label{Prop E type fini abelien}
Let~$M^{der}$ be the derived group of~$M$ and~$M^{ab}=M/M^{der}$ its maximal abelian quotient.
We consider the quotient map
\[
\pi:M(\QQ_S)\to M^{ab}(\QQ_S).
\]
\begin{enumerate}
\item \label{Prop E 1} If~$s$ is of $S$-Tate type, then the centre~$Z(U_S)$ of~$U_S$ is contained in the centre of~$M(\QQ_S)$.
\item \label{Prop E 2} If~$s$ is of $S$-Tate type and $S$-simplicity holds, then we have~$\pi(Z(U_S))$ is open in $\pi(U_S)$.
\item \label{Prop E 3} If~$E$ is of finite type over~$\QQ$, then~$\pi(U_S)$ is open in~$M^{ab}(\QQ_S)$.
\end{enumerate}
In particular, when $M$ is abelian (i.e. the point $s$ is special), the $S$-Mumford-Tate
property holds.
\end{prop}
The last remark is a straightforward consequence of the reciprocity law for canonical models of Shimura varieties.
We will rely on this law for the proof of~\eqref{Prop E 3}.
\begin{proof}
Let us prove~\eqref{Prop E 1}. The center of~$U_S$ is the intersection of~$U_S$ with its centraliser~$Z_{G_S}(U_S)$.
By the $S$-Tate property we have~$Z_{G_S}(U_S)=Z_{G_S}(M)$. We also have~$U_S\subseteq M(\QQ_S)$. It follows
\[U_S\subseteq M(\QQ_S)\subseteq M(\QQ_S)\cap Z_{G_S}(M)\]
but the latter is just the centre~$Z(M)$ of~$M$.

We now prove~\eqref{Prop E 2}. 
Let $x \in \pi(U_S)$. Write $x=\pi(y)$ for $y \in U_S$. 
Since semisimplicity holds, there exists an integer $n$ (depending only on $U_S$ but not $y$)
such that 
$$
y^n = z\cdot t
$$
with $z \in  Z_{G_S}(U_S)$ (by ~\eqref{Prop E 1}) and $t \in  U_S^{der} \subset M^{der}(U_S)$.
Thus $x^n = \pi(y^n) = \pi(z) \in \pi(Z(U_S))$. This proves ~\eqref{Prop E 2}.
 
We prove~\eqref{Prop E 3}.
As~$M$ is reductive, the restriction of~$\pi$ to~$Z(M)$ is an isogeny. Since $S$ is a finite set of primes,  it follows that the induced map~$Z(M)(\QQ_S)\to M^{ab}(\QQ_S)$
has finite kernel and an open image.
It will hence suffice to prove that the image of~$U_S\cap Z(M)$ is open in~$M^{ab}(\QQ_S)$. 

In the Shimura variety\footnote{Of dimension~$0$, related to the space of connected components of~$\Sh(G,X)$.}, associated with~$M^{ab}$ every point is a special point.
So is the image of~$s$. The associated Galois representation is
\[
 \pi\circ\rho_{s,S}: \Gal(\overline{E}/E)\to M(\QQ_S)\to M^{ab}(\QQ_S).
\]
By the reciprocity law for the canonical model of the Shimura variety associated with~$M^{ab}$, this representation
factors through 
\[
 \Gal(\overline{\QQ}/E(M,X_M))\to M^{ab}(\QQ_S).
\]
and the image of the latter is open (\cite[composantes connexes]{Deligne}).
As~$E$ is of finite type, the algebraic closure~$F$ of the number field~$E_ME(M,X_M)$ 
in~$E$ is finite. Then the image of~$\Gal(\overline{E}/E)$ in~$\Gal(\overline{\QQ}/E_M)$
is the open subgroup~$\Gal(\overline{\QQ}/F)$. It follows that the image in~$ M^{ab}(\QQ_S)$ of~$\Gal(\overline{E}/E)$
is open  (of index at most~$[F:E_M]$) in that of~$\Gal(\overline{\QQ}/E_M)$. This concludes the proof of the proposition.
\end{proof}

\subsubsection*{$S$-Tate property and $S$-algebraicity.}

We prove here that when $E$ is of finite type over $\QQ$, the $S$-Tate property implies the $S$-algebraicity.
Let~$\Fm$ be the Lie algebra
of~$M_{\QQ_S}$, let~$\Fu$ be the Lie algebra of~$U_S$, a $\QQ_S$-Lie subalgebra 
of~$\Fm$. Let~$H$ be the $\QQ_S$-algebraic envelope of~$U_S$, and~$\Fh$ its Lie algebra.

We refer to \cite[BorelLAG] and \cite{Serre} for the following.

\begin{prop} The following are equivalent:
\begin{itemize}
 \item the subgroup~$U_S$ of~$H(\QQ_S)$ is open (for the $S$-adic topology);
 \item the Lie algebra~$\Fu$ of~$U_S$ is an \emph{algebraic} Lie subalgebra of the Lie algebra~$\Fm$ of~$M$, in the sense of Chevalley as in~\cite[II.\S7]{BorelLAG};
 \item the Lie algebras~$\Fu$ of~$U_S$ and~$\Fh$ of~$H$ are the same:~$\Fu=\Fh$.
\end{itemize}
If these properties hold we say that~$U_S$ is an \emph{algebraic Lie subgroup}.
\end{prop}

\begin{prop} Assume that Lie algebras~$\Fu$ is of $S$-Tate type and satisfies~$S$-semisimplicity.
If~$E$ is an extension of finite type of~$\QQ$, then the Lie subgroup~$U_S$ is algebraic in the sense of the previous proposition, that is~$s$ satisfies the~$S$-algebraicity in the sense of definition~\ref{ProprietesGaloisiennes}\,\eqref{Algebraicity}.
\end{prop}

\begin{proof}We first use the $S$-semisimplicity to note that the adjoint action of~$\Fu$ on itself is semisimple,
that is~$\Fu$ is a reductive Lie algebra by the definition used in~\cite[I \S6.4 D\'{e}f.~4]{Lie1}. 
It follows that  we may decompose~$\Fu$ as a direct sum~$\Fu=[\Fu,\Fu]+\Fz$ of its derived Lie algebra~$[\Fu,\Fu]$ and its centre~$\Fz$, by~\cite[Cor. (b) to Prop. 5]{Lie1}. It is enough to show that both~$[\Fu,\Fu]$ and~$\Fz$ are algebraic by~\cite[Cor. 7.7 (1) and (3)]{BorelLAG}. The derived Lie algebra~$[\Fu,\Fu]$ is algebraic by~\cite[Cor. 7.9]{BorelLAG}. 

It remains to prove that~$\Fz$ is an algebraic Lie subalgebra. 
We use the $S$-Tate type property, to note that the centraliser of~$U$ is included
in the  centraliser of~$M$. We infer~$\mathfrak{z}\subseteq\mathfrak{z}_M(\QQ_S)$. As~$E$ is of finite type, 
we may apply Prop.~\ref{Prop E type fini abelien}.
we have seen  that~$U_S$ contain an open subgroup of the centre of~$M(\QQ_S)$. We have
conversely~$\mathfrak{z}\supseteq \mathfrak{z}_M(\QQ_S)$. Finally,~$\mathfrak{z}$ is the algebraic lie subalgebra~$\mathfrak{z}_M$.
\end{proof}
\subsubsection*{N.B.:} This algebraicity statement is similar to Bogomolov's algebraicity result~\cite[Th.~1]{Bogo} for abelian varieties. The reduction to the case of abelian Lie algebras is very similar. We rely on $S$-Tate property and the theory of canonical models to treat the abelian case.

For similar algebraicity results see~\cite[133. Th. p.~4]{SerreOEuvres4}, and notably its subsequent Corollary.


\subsubsection{Characterisation of the $S$-Shafarevich property.}
We prove that~$S$-Shafarevich property~\ref{ProprietesGaloisiennes}\,\eqref{S-Sha} is equivalent to the conjonction of the $S$-semi-simplicity property~\ref{ProprietesGaloisiennes}\,\eqref{S-semisimple} and a weakening or the~$S$-Tate property (as defined in~\ref{ProprietesGaloisiennes}\,\eqref{S-Tate}). This is amounts to a group theoretic characterisation of the~$S$-Shafarevich property, which is essential for proving the main theorems of this paper and which, we believe, is also of independent interest.

%



\begin{prop} \label{charSha} Let~$s$ be a point of~$\Sh_K(G,X)$ and~$E$ a field of definition of~$s$ such that the associated $S$-adic monodromy representation~\[\rho:\Gal\left(\ol{E}/E\right)\to M(\AAA_f)\cap G_S\] is defined, and the image~$U_S$ of~$\rho$ is Zariski-connected. 
Let~$Z_{G_S}(M)$ and~$Z_{G_S}(U_S)$ denote the centraliser in~$G_S$ of the Mumford-Tate group~$M$ of~$s$ and of the image~$U_S$ of~$\rho$ respectively.

The point~$s$ is of $S$-Shafarevich type if and only if it satisfies $S$-semisimplicity and, furthermore,~$Z_{G_S}(M) \backslash Z_{G_S}(U_S)$ is compact.

In particular the $S$-Shafarevich property is implied by the conjunction of $S$-semisimplicity and $S$-Tate properties.
\end{prop}
We will split the proof into several steps.
Let us consider a sequence $s_n = \overline{(h,g_n)}$ with $g_n \in G_S$
of points  in the~$S$-Hecke orbit~$\mathcal{H}_S(s)$ of~$s$. 

We start with an easy Lemma about the homogeneous structure of the $S$-Hecke orbit~$\Hcal_S(s)$.
\begin{lem}\label{Lemme Hecke orbite Z}We embed~$Z_G(M)(\QQ)$ into~$G(\AAA_f)$ (at the finite places \emph{only}) and~$G(\QQ)$ into~$G(\AAA)$. Then the application
\begin{multline*}
Z_G(M)(\QQ)\backslash Z_G(M)(\QQ)\cdot G_S\cdot K/K
\\\to \Hcal_S(s)=G(\QQ)\backslash G(\QQ)\cdot (\{h\}\times G_S)\cdot K/K,
\end{multline*}
which maps the double class of~$g$ to~$\ol{(h,g)}$, is a bijection.
\end{lem}
\begin{proof} The surjectivity is immediate, by the very definition of~$\Hcal_S(s)$. We prove the injectivity. We start with the identity~$\ol{(h,g)}=\ol{(h,g')}$. Equivalently there is~$q$ in~$G(\QQ)$ and~$k$ in~$K$ such that
\[
q\cdot(h,g)\cdot k=(h,g').
\]
From~$qh=h$, we infer that~$q$ is in the stabiliser of~$h$. As~$q$ is in~$G(\QQ)$, it belongs to the biggest~$\QQ$-subgroup in the stabiliser of~$h$, which is~$Z_G(M)(\QQ)$. We conclude by observing
\[
q\cdot g\cdot k= g'.\qedhere
\]
\end{proof}

In the following Lemma, we show how the $S$-simplicity hypothesis can be used to work out a group theoretic condition on $g_n$ for~$s_n$ to be defined over $E$.
\begin{lem}\label{stab}
Assume that the $S$-simplicity holds.

There exists a compact subset $C \subset G_S$ such that if a point~$\ol{(h,g)}$ of~$\Hcal_S(s)$ is defined over~$E$, then
$$
g \in Z_{G_S}(U_S)\cdot C.
$$
\end{lem}
\begin{proof} The point~$\ol{(h,g)}$ is defined over~$E$ if and only if, for every element~$u$ of~$U_S$, we have
\[
\overline{(h, u g)} = \overline{(h, g)}.
\]
This means that there exists a $q$ in $G(\QQ)$ and $k$ in $K$, depending on~$u$, such that
\begin{eqnarray*}
q h = h,&\text{ i.e. }&q \in Z_G(M)(\QQ),\\
\text{ and }q u g = g k,&\text{ i.e. }&qu = g k g^{-1}.
\end{eqnarray*}
Thus~$qu$ belongs to the group~$gKg^{-1}$. Any power~$(qu)^n$ belongs to the same group.
As~$q$ centralises~$U_S$, we have~$(qu)^n=u^nq^n$. It follows
\[
q^n=u^{-n}\cdot gk^ng^{-1}.
\]
In particular all powers of~$q$ belongs to the compact set~$U_S\cdot gKg^{-1}$. They also
belong to the discrete set~$Z_G(M)(\QQ)$ in~$G(\AAA_f)$. It must be that~$q$ is a torsion element of~$Z_G(M)(\QQ)$.

Actually all torsion elements of~$Z_G(M)(\QQ)$ satisfy~$q^N=1$ for a uniform order~$N>0$: we may embed~$Z_G(M)$ in a linear group~$GL(D)$ and apply Lemma~\ref{Lemme phi}. We hence have
\begin{equation}\label{eq transporteur}
u^N=q^Nu^N=gk^Ng^{-1},\text{ whence }g^{-1}u^Ng\in K.
\end{equation}
Let~$V=\{u^N|u\in U\}$. This is a neighbourhood of the neutral element in~$U$, as the~$N$-th power map has a non-zero differential at the origin: it is the multiplication by~$N$ map on the Lie algebra. It follows that~$V$ is Zariski dense in~$U$ (recall that~$U$ is Zariski connected).

We deduce from~\eqref{eq transporteur} above that~$g$ belongs to the transporteur, for the conjugation right-action, of~$V$ to~$K$
\[
T=\{t\in G_S|t^{-1}Vt\subseteq K\}.
\]

The Zariski closed subgroup generated by~$V$ is the same as the one generated by~$U$, and is a reductive group by hypothesis. This is the essential hypothesis we need to invoke~\cite[Lemma D.2]{RU},  according to which there exists a compact subset $C$ of $G_S$ such that
\[
g\in T \subseteq Z_{G_S}(U_S)\cdot C. 
\]
We are done, but from the fact that actually \emph{loc. cit.} works only for one ultrametric place at a time. But arguing with the projections~$G_p$, $U_p$, $V_p$, $K_p$ and~$T_p$ of~$G_S$, $U_S$, $V$,~$K$ and~$T$ in~$G_p$ at a some place~$p$ in~$S$, we can prove as above that
\[
T_p\subseteq \{t\in G_p|t^{-1}V_p t\subseteq K_p\}\subseteq Z_{G_p}(U_p)\cdot C_p
\] 
for a compact subset~$C_p$ of~$G_p$, and conclude, with~$C$ the product compact set
\[
T\subseteq \prod_{p\in S} T_p\subseteq\prod_{p\in S} Z_{G_p}(U_p)\cdot C_p=Z_{G_S}(U_S)\cdot \prod_{p\in S} C_p=Z_{G_S}(U_S)\cdot C.\qedhere
\]
\end{proof}
The following standard fact was used in the preceding proof.
\begin{lem}\label{Lemme phi} 
Consider a general linear group~$GL(D,\QQ)$ over the field~$\QQ$ of rational numbers. Then there is an integer~$N(D)$ such that for any~$g$ in~$GL(D,\QQ)$ of finite order its power~$g^{N(D)}$ is the neutral element.
\end{lem}
\begin{proof}Let~$g$ be a torsion element. Every complex eigenvalue of~$g$ is some root of unity~$\zeta$. Let~$d$ be the order of~$\zeta$. As the cyclotomic polynomials are irreducible over~$\QQ$, it must be that~$g$ has at least~$\phi(d)$ eigenvalues, the algebraic conjugates of~$\zeta$. We hence have~$\phi(d)\leq D$.

It is known that~$\phi(n)$ diverges to infinity as~$n$ diverges to infinity (one has~$n^{1-\varepsilon}=o(\phi(n))$ for instance).
There is a largest integer~$N$ such that~$\phi(N)\leq D$.

We have necessarily~$d\leq N$. It follows~$d|N!$, and hence~$\zeta^{N!}=1$. The only eigenvalue of~$g^{N!}$ is~$1$.
The power~$g^{N!}$ is unipotent. It is also of finite order, and we must have that~$g^{N!}$ is the neutral element. TWe can take~$N(D)=N!$.
\end{proof}
We now prove one implication in the second part of Proposition~\ref{charSha}.
\begin{lem}
Assume now that $S$-simplicity holds and that 
$Z_{G_S}(M)$ is cocompact in~$Z_{G_S}(U_S)$.
Then~$s$ is of $S$-Shafarevich type.
\end{lem}
\begin{proof}
Let $\overline{(h,g_n)}$ be a sequence of points, in the $S$-Hecke orbit~$\Hcal_S(s)$, defined over a finite extension~$F$ 
of $E$. Our aim is to show this sequence can take at most finitely many distinct values.
After replacing $U_S$ by an open subgroup of finite index, we may assume that~$F=E$, which translates into the property
$$
\forall u\in U_S,~\overline{(h, u\cdot g_n)} = \overline{(h,g_n)}
$$
for all $n$.

Let us write
\begin{subequations}\label{Abbr Z}
\begin{align}
Z&=Z_G(M),&
Z(\QQ_S)&=Z_{G_S}(M),
\end{align}
\begin{equation}
Z(\QQ)=Z_G(M)(\QQ)\subseteq Z(\AAA_f)=Z_G(M)(\AAA_f)
\end{equation}
\end{subequations}

Lemma \ref{stab} shows that elements $g_n$ are contained in 
$T=Z_{G_S}(U_S)\cdot C$ for some compact subset~$C$ of 
$G_S$.
By hypothesis,~$Z_{G_S}(M)\backslash Z_{G_S}(U_S)$ is compact.
By the adelic version of Godements's compactness criterion 
 it is also true that $Z(\QQ) \backslash Z(\AAA_f)$ is compact as well, as~$Z$ is~$\RR$-anisotropic up to the centre of~$G$.

We remark that~$Z(\AAA_f)$ normalises~$Z_{G_S}(U_S)$. As the latter is a place by place product it can be checked place by place: at places in~$S$ the projection of~$Z(\AAA_f)$ is contained in~$Z_{G_S}(M)$ which is itself contained in~$Z_{G_S}(U_S)$; at other places~$Z_{G_S}(U_S)$ has only trivial factors. We hence have a homomorphism
\begin{equation}\label{conversion S vs adelic}
Z_{G_S}(M)\backslash Z_{G_S}(U_S)\longrightarrow{}
Z(\AAA_f)\backslash Z_{G_S}(U_S)\cdot Z(\AAA_f),
\end{equation}
which is surjective, with compact source, hence has compact image.

Incorporating with the compactness of~$Z(\QQ)\backslash Z(\AAA_f)$ we infer the compactness of
\[
Z(\QQ)\backslash Z_{G_S}(U_S)\cdot Z(\AAA_f)
\]
and follows the compactness of the subset
\[
Z(\QQ)\backslash  Z(\AAA_f)\cdot Z_{G_S}(U_S)\cdot C= Z(\QQ)\backslash  Z(\AAA_f)\cdot T
\]
of~$Z(\QQ)\backslash G(\AAA_f)$ (we note that this quotient is separated, as~$Z(\QQ)$ is discrete in~$G(\AAA_f)$).

As~$K$ is open, we deduce that 
\[
 Z(\QQ)\backslash  Z(\AAA_f)\cdot T\cdot K/K
\]
is finite. To sum it up we have
\[
Z(\QQ)\cdot g_n \cdot K\in Z(\QQ)\backslash  Z(\AAA_f)\cdot T\cdot K/K\subseteq  Z(\QQ)\backslash Z(\QQ) T\cdot K/K.
\]
Now, the double coset on the right characterises~$\ol{(h,g_n)}$, by Lemma~\ref{Lemme Hecke orbite Z}.
Finally~$\ol{(h,g_n)}$ can take at most~$\#Z(\QQ)\backslash Z(\QQ) T\cdot K/K$ distinct values.\qedhere
%
%

%
%

\end{proof}

To prove the other inclusion, let us first prove that 
the $S$-Shafarevich property implies $S$-semisimplicity.
This is done in the following lemma.

\begin{lem}
The $S$-Shafarevich property implies the $S$-semisimplicity property, namely that the group $H_S$ is reductive.
\end{lem}
\begin{proof} We can argue place by place: indeed the place by place Shafarevich hypothesis is weaker than~$S$-Shafarevich property, and the reductivity of~$H_S$ can be checked place by place. We may assume for simplicity that~$S$ consists of only one finite place.

We will prove the contrapositive statement, namely that for non reductive~$H_S$ the $S$-Shafarevich property cannot hold.
Let us denote the unipotent radical of~$H$ by~$N_H$. That~$H_S$ is reductive means that~$N_H$ is trivial. We assume it is not the case.

We apply \cite[Proposition 3.1]{BorelTits} to the unipotent subgroup~$N_H \subset G_S$.
There exists a parabolic subgroup $P$ of $G_S$ such: that $N_H$ is
contained in the unipotent radical~$N_P)$ of $P$ and the normaliser of
$N_H$ in $G_S$ is contained in $P$. In particular, $H_S$ is contained in $P$.

By \cite[Prop 8.4.5]{Springer} there exists a cocharacter 
$$
y : {\bf G}_m \lto G_S
$$
of $G_S$ over $\QQ_S$  such that
$$
P = \{ g \in G_S : Ad_{y(t)}(g) \text{ converges as } t \lto \infty \} 
$$
and (cf. \cite[Th. 13.4.2(i)]{Springer}, \cite[\S2.2 Def. 2.3/Prop. 2.5]{GIT})
$$
N_P = \{ g \in G_S : Ad_{y(t)}(g) \text{ converges to~$e$ as } t \lto \infty \}.
$$
Moreover, the centraliser of~$y$ in~$P$ is a Levi factor~$L$ of~$P$. It follows that for all~$p=\lambda\cdot n$ in~$P$,
with~$l$ in~$L$ and~$n$ in~$\Radu(P)$, the limit~$\lim_{t\to\infty} Ad_{y(t)}(p)$ is the factor~$l$.

We will contradict the $S$-Shafarevich property by showing that the family of the~$\ol{(h,y(t))}$, as~$t$ diverges to infinity,
\begin{itemize}
\item describes infinitely many points in the $S$-Hecke orbit of~$s$,
\item and that these points are all defined on a common finite extension of~$E$.
\end{itemize}

We address the first statement.

By the non triviality of~$N_H$, we may pick an element~$u$ in~$N_H$ distinct from~$e$. It follows that 
the conjugacy class~$C(u)$ of~$u$ in~$G_S$ does not contain~$e$.
As~$Z_G(M)$ centralises~$H_S$ and its subgroup~$N_H$, the orbit map at~$u$ for the 
conjugation action factors through~$G_S/Z_{G_S}(M)$. We have a map
\[
c:G_S/Z_{G_S}(M)\xrightarrow{gZ_{G_S}(M)\mapsto Ad_g(u)} C(u).
\]

We can deduce by contradiction that~$y(t)$ is not bounded modulo~$Z_{G_S}(M)$ as~$t$ diverges to~$\infty$.
Assume not. Then~$y(t)Z_{G_S}(M)$ would have some accumulation point~$gZ_{G_S}(M)$. Hence~$c(y(t)Z_{G_S}(M))$ would have the accumulation point~$c(gZ_{G_S})$, which belongs to~$C(u)$ and hence is distinct from~$e$. This contradicts
the fact that~$c(y(t)Z_{G_S}(M))$ converges to~$e$.

The Hecke orbit of~$s$ can be identified with~$Z_{G_S}(M)(\QQ)\backslash G_S /K_S$ through the quotient of the map~$g\mapsto \ol{(s,g)}$ on~$G_S$. Let us claim that~$Z_{G_S}(M)(\QQ)y(t)K_S$ describes infinitely many cosets. It is sufficient that~$Z_{G_S}(M)y(t)K_S$ does so. If not, then~$Z_{G_S}(M)y(t)$ would be contained in finitely many right~$K_S$ orbits, that is in a bounded set of $Z_{G_S}(M)$-cosets, which cannot be, as we already proved. This proves the statement.

We address the second statement, investigating the field of definition of these $S$-Hecke conjugates of~$s$. We use that the extension of definition of a point~$\ol{(s,g)}$ is associated with the finite quotient~$gU_Sg^{-1} K_S/K_S$ of~$\Gal(\ol{E}/E)$.

As~$U_S$ is topologically of finite type, it will be sufficient to show that~$\# y(t)U_Sy(t)^{-1} K_S/K_S$ is bounded as~$t$ diverges to~$\infty$. It will even be sufficient that~$y(t)U_Sy(t)^{-1}$ remains in a bounded subset~$C$ of~$G_S$, as then we have the bound~$\# y(t)U_Sy(t)^{-1} K_S/K_S\leq \# CK_S/K_S$. 

On~$P$ the family of functions~$Ad_{y(t)}$ converges simply to the projection onto the Levi factor~$L$ of~$P$.
Let us prove the claim that this convergence is uniform on compacts subsets of~$P$, including~$U_S$. As~$Ad_{y(t)}$
acts on factor of the Levi decomposition~$P=L M$ of~$P$ separately, it will suffice to argue for~$L$ and~$\Radu(P)$
separately. 
\begin{itemize}
\item It is immediate for~$L$, on which~$Ad_{y(t)}$ is the identity, independently from~$t$.
\item  We turn to the unipotent group~$N_P$. We may argue at the level of Lie algebras, on which the corresponding action~$ad_{y(t)}$ is linear. 
Yet again we may argue separately, this time with respect to the decomposition into eigenspaces. On a given eigenspace,~$ad_{y(t)}$ acts by a negative power of~$t$, which converges uniformly to~$0$ as~$t$ diverges to~$\infty$ on any bounded subset.
\end{itemize} 
This proves the claim. 

We conclude that~$Ad_{y(t)}$ is a uniformly bounded family on~$U_S$ as~$t$ diverges to~$\infty$: there is a bounded set~$C$ that contains~$Ad_{y(t)}(U_S)$ for big enough~$t$. We obtained bounded subset that we sought. This proves the second statement. 

We have proven that the~$S$-Shafarevich property cannot hold.
%
%
%
%
\end{proof}

We can now conclude the proof of the implication.

\begin{lem}
The $S$-Shafarevich property implies that~$Z_{G_S}(M)$ is cocompact in~$Z_{G_S}(H_S)$.
\end{lem}
\begin{proof}
Assume  that $Z_G(M)\backslash Z_{G_S}(U_S)$ is not compact, and let us disprove the $S$-Shafarevich property.
Possibly substituting~$E$ with a finite extension thereof, we may assume~$U_S\subseteq K$. 
We will prove that
\[\Hcal=\{\ol{(h,z)}|z\in Z_{G_S}(U_S)\}\]
is an infinite set of points defined over~$E$.

Let~$\ol{(h,z)}$ be such a point. For~$\sigma$ in~$\Gal(\ol{E}/E)$ with image~$u$ in~$U_S$ we have
\[
\sigma\left(\ol{(h,z)}\right)=\ol{(h,uz)}=\ol{(h,zu)}=\ol{(h,z)},
\]
by definition of~$U_S$, by the fact that~$z$ commutes with~$U_S$, and that~$U_S$ is contained in~$K$ respectively.
It follows that these points are defined over~$E$.

By Lemma~\ref{Lemme Hecke orbite Z} we have a bijection
\[
\Hcal\simeq Z_G(M)(\QQ)\backslash Z_G(M)(\QQ)\cdot Z_{G_S}(U_S)\cdot K/K.
\]
We need to to prove this is an infinite set. 
As the following map of double quotients
\[
Z(\QQ)\backslash Z(\QQ)\cdot Z_{G_S}(U_S)\cdot K/K
\longrightarrow
Z(\AAA_f)\backslash Z(\AAA_f)\cdot Z_{G_S}(U_S)\cdot K/K
\]
is a surjection, it is sufficient to prove its image is infinite. As~$K$ is compact,
it is enough that
\begin{equation}\label{quotient Z A}
Z(\AAA_f)\backslash Z(\AAA_f)\cdot Z_{G_S}(U_S)
\end{equation}
be unbounded in~$Z(\AAA_f)\backslash G(\AAA_f)$. 

Let us accept for now  that the the map
\begin{equation}\label{closed immersion}
Z(\QQ_S)\backslash G(\QQ_S)\longrightarrow Z(\AAA_f)\backslash G(\AAA_f),
\end{equation}
induced by the closed immersion~$\QQ_S\to\AAA_f$, is itself a closed immersion. Moreover~$Z_{G_S}(U_S)$ is closed in~$G(\QQ_S)$ and contains~$Z(\QQ_S)$. Hence the group~$Z(\QQ_S)\backslash Z_{G_S}(U_S)$ embeds as a closed subset of~$Z(\QQ_S)\backslash G(\QQ_S)$. It is also non compact by hypothesis. Its image in~$Z(\AAA_f)\backslash G(\AAA_f)$ is closed and non compact. It is hence unbounded. But this image is~\eqref{quotient Z A}. This concludes

We prove now that~\eqref{closed immersion} is a closed immersion. It is certainly the case of
\[
(Z\backslash G)(\QQ_S)\longrightarrow
(Z\backslash G)(\AAA_f)
\]
as it induced by the closed immersion~$\QQ_S\to\AAA_f$ (we may embed~$Z/G$ as a closed subvariety in an affine space). Moreover the map
\[
Z(\QQ_S)\backslash G(\QQ_S) \longrightarrow (Z\backslash G)(\QQ_S)
\]
is the kernel of the continuous map~$ (Z\backslash G)(\QQ_S)\to H^1(\QQ_S;Z)$, 
hence a closed immersion. The image~$F$ of a closed subset of~$Z(\QQ_S)\backslash G(\QQ_S)$
in~$(Z\backslash G)(\AAA_f)$ is hence closed. The inverse image of~$F$ by the continuous map~$Z(\AAA_f)\backslash G(\AAA_f)\to (Z\backslash G)(\AAA_f)$ is a fortiori closed. The map~\eqref{closed immersion} is then a closed map.
Clearly, it is furthermore injective. It is finally a closed immersion. This concludes.\qedhere

\end{proof}

This finishes the proof of Proposition \ref{charSha}.

\subsection{Shimura varieties of abelian type.}

In this section we show that when $E$ \emph{is of finite type} over $\QQ$
and $(G,X)$ is a Shimura datum of abelian type, all the properties
except possibly $S$-Mumford-Tate hold.

\begin{prop}
Assume that $\Sh_K(G,X)$ is a Shimura variety of abelian type
and that $E$ is a field of finite type.

Then the $S$-semisimplicity, $S$-Tate (and hence $S$-Shafarevich)
and $S$-algebraicity  hold for all points of 
 $\Sh_K(G,X)(E)$.
\end{prop}

\begin{proof}
By definition of a Shimura variety of abelian type, there exists 
a Shimura subdatum $(G',X) \subset (GSp_{2g}, \HH_g)$
with a central isogeny $\theta \colon G' \lto G$.

There exists a compact open subgroup $K' \subset G'(\AAA_f)$ such that
$\Sh_{K'}(G',X)$ is a subvariety of $\cA_{g,3}$ (the fine moduli scheme 
of abelian vareities with level $3$ structure) and there is a finite morphism
$\Sh_{K'}(G',X) \lto \Sh_K(G,X)$.

For the purposes of proving the $S$-Tate property we may assume that
all our Shimura varieties are defined over $E$.

Let $x = \overline{(s,t)} \in \Sh_{K}(G,X)(E)$ and $x'$ a point 
of $\Sh_{K'}(G',X)$ over a finite extension of $E$.
Without loss of generality in the proof of the $S$-Tate property 
we replace $E$ by this finite extension and hence assume that $s'$ is defined over $E$.

The point $s'$ corresponds to an abelian varietey $A$ defined over $E$.
Let $V_S(A) = \prod_{l \in S} V_l(A)$ be the product of $l$-adic Tate modules attached to $A$ for $l \in S$. The module $V_l(A)$ is endowed 
with an action of $\Gal(\ol{E}/E)$ and with a symplectic action of 
$\GSp_{2g}(\QQ_S)$. Following the arguments of Remark 2.8
of \cite{UY1}, we see that the action of $\Gal(\ol{E}/E)$ is given by 
the representation 
$$
\rho'_S \colon \Gal(\ol{E}/E) \lto G'_S \subset \GSp_{2g}(\QQ_S).
$$

Note that in \cite{UY1} there the authors suppose the field $E$ to be a number field and $S$ to consist of one prime. However, all arguments adapt verbatim in our situation.

We denote by $M'$ the Mumford-Tate of $x'$ and by $H'_S$
the image of $\rho'_S$.

By Falting's theorem (Tate conjecture for abelian varieties),
 the group $H'_S$ is reductive. Note that Falting's theorem holds for abelian
varieties over finitely generated fields over $\QQ$ (see Chapter  VI of \cite{FW}).
 
 Again, by the Tate conjecture, we have 
 $$
 Z_{GSp_{2g, S}}(M') = Z_{GSp_{2g, S}}(H'_S)
 $$
 and therefore 
 $$
 Z_{G'_S}(M') = Z_{G'_S}(H'_S).
 $$
 Let $M$ be the Mumford-Tate group of $x$ and let $H_S$ be Zariski closure of the image 
 of $\rho_S$, the representation of $\Gal(\ol{E}/E)$ attached to $x$.
 
 Recall that we have a central isogeny $\theta \colon G' \lto G$.
 We naturally have
 $$
 M = \theta(M'), \quad H_S = \theta(H'_S).
 $$ 
In particular, $H_S$ is semisimple. 
 
 Since $\theta$ is a central isogeny and hence 
 commutes with conjugation, the equality $Z_{G'_S}(M') = Z_{G'_S}(H'_S)$.
\end{proof}

\subsection{Dependence on the field $E$.}

We have seen that for a Shimura variety of abelian type and
when $E$ is of finite type all properties~\ref{ProprietesGaloisiennes}
(except $S$-Mumford-Tate) hold.

In this section we will show that these properties fail even when $(G,X)$ is
of abelian type when the field $E$ is not of finite type.

Let $A$ be an elliptic curve over $\QQ$ without complex multiplication.
By a celebrated theorem of Serre, the image of the adelic representation of $\Gal(\overline{\QQ}/\QQ)$ attached to $A$ is open in $\GL_2(\widehat{\ZZ})$.
For any $S$, all our properties hold for $A$ over $\QQ$ 
(or any finite type extension of $\QQ$).

Choose a prime $l$ such that $\rho_l$ surjects onto $\GL_2(\ZZ_l)$.
Take $S= \{ l \}$.
Let $B$ be the standard Borel in $\GL_2(\ZZ_l)$ (upper triangular matrices)
and let $E$ be an extension corresponding to the subgroup $\rho_l^{-1}(B)$ of $\Gal(\overline{\QQ}/\QQ)$.
The extension $E$ is not of finite type over $\QQ$.

We immediately see that $S$-Mumford-Tate property does not hold
for $A/E$ (the image of Galois is $B$). The $S$-semisimplicity also fails
since the Zariski closure of $B$ is not reductive.
The $S$-Tate property holds however.
Indeed the centraliser of $B$ is the centre of $\GL_2$ which is of course
also the centraliser of $\GL_2$ in itself.

Finally, let us see directly that the $S$-Shafarevich property fails.
Write $\overline{(s, 1)}$ the point of $\Sh_{\GL_n(\widehat{\ZZ})}(\GL_2,\HH^{\pm})$ corresponding to $A$.

Consider elements
$g_n = \left(\begin{smallmatrix}l^{-n}& 0\\0&1\end{smallmatrix}\right)$ for $n > 0$.

Note that
$$
g_n^{-1} B g_n \subset B
$$

and therefore the sequence of points $\overline{(s,g_n)}$ is defined over 
$E$. This sequence of points is obviously infinite - it corresponds to all elliptic curves isogeneous (over $\CC$) to $A$ by a cyclic isogeny of
degree $l^n$.

We now give an example where $S$-algebraicity fails.
In the previous situation, consider the subgroup of $\GL_2(\ZZ_l)$
defined by 
$\left(\begin{smallmatrix}e^x& x\\0&e^x\end{smallmatrix}\right)$
where $x \in l^2 \ZZ_l$.

Note that this group is Zariski dense in the group
$\left(\begin{smallmatrix}a& b\\0&a\end{smallmatrix}\right)$  where $a,b \in \ZZ_l$.
However it is not open in this group and therefore $S$-algebraicity fails.

Note that in the previous example semi-simplicity did not hold. However,
we can construct an example where $S$-semi-simplicity does hold, but
the $S$-algebraicity fails.

Let $\alpha$ be the Liouville number, 
$\alpha = \sum_{n=0}^{\infty} l^{n!}$. This is an element of $\ZZ_l$ which is transcendental over $\QQ$. We refer to \cite{AM} and references therein for more details.

Consider the subgroup 
$\left(\begin{smallmatrix}e^{x}& 0\\0&e^{\alpha x}\end{smallmatrix}\right)$  where $x \in \ZZ_l$ and as before $E$ the extension corresponding to this subgroup. The Zariski closure is the diagonal torus therefore 
the $S$-semisimplicity holds.
However the $S$-algebraicity fails.
This example is analogous to the one given after Theorem 2.3 in \cite{Bost}.


\section{The $S$-Arithmetic lift and Equidistribution}\label{Reduction}

In this section we start working towards proving Theorem~\ref{Theoreme1}.
Our aim here is to translate our problem (i.e. Theorem~\ref{Theoreme1}) about the probabilities~$\mu_n$ on~$\Sh_K(G,X)$ into a problem of equidistribution 
on an~$S$-arithmetic homogeneous space of a semisimple algebraic group over~$\QQ$, of the kind studied~\cite{RZ}, in order to apply the main theorem thereof. More precisely, we will use the maps~$\pi_G$ of~\eqref{piG} and~$\pi$ of~\eqref{pider}.
We warn here that the map $\pi$ we construct in this section is not the same as the one considered in the previous section.

We work with a point~$s = \overline{(h,t)}$ of $\Sh_K(G,X)$ and 
its $S$-Hecke orbit~$\Hcal_S(s)$ as in the statement of Theorem~\ref{Theoreme1}.

We also have a sequence of points~$s_n = \ol{(h,t\cdot g_n)}$ of~$\Hcal_S(s)$
with some $g_n \in G_S$.
The elements~$g_n$ are not necessarily uniquely defined by the points $s_n$ and will actually be subject to modification, without changing~$s_n$, in the course of the proof.

For any~$g$ in~$G_S$, we may change the representative~$(h,t)$ of~$s$ into the representative~$(h,t\cdot g)$, provided that we change accordingly each~$g_n$ into~$g^{-1}g_n$. Neither the~$S$-Hecke orbit~
$\Hcal_S(s)$, nor the set of points~$\{s_n\}$ are changed by such substitutions.  The group~$U_S$ is left unchanged too. Hence the $S$-Shafarevich property of~$s$ remains valid under this substitution of~$s$ with~$\overline{(h,tg)}$.

Lastly, since the conclusion of Theorem~\ref{Theoreme1} we may
``extract a subsequence'', we may, whenever necessary, replace~$(g_n)_{n\geq0}$ by a subsequence.

For technical reasons we will need to assume that $H_S \cap G^{der}$
to be $\QQ_S$-Zariski connected. 
It is sufficient that~$U_S\cap G^{\der}(\QQ_S)$ be  contained in~$\left( H_S\cap G^{\der}_{\QQ_S}\right)^0(\QQ_S)$.
This can be achieved by passing to a subgroup of finite index in~$U_S$, that is to a finite extension of~$E$.
We note that this assumption will still be fulfilled after passing again to a finite extensions of~$E$. We
recall that the statement of Theorem~\ref{Theoreme1} allows passing to finite extensions of~$E$.

\subsection{The reductive $S$-arithmetic lift.}

Since the equidistribution theorems of \cite{RZ} apply to $S$-arithmetic homogeneous
spaces, we need to ``lift'' our situation, from the base Shimura variety of level~$K$ to such a space.

In this section we construct a map~$\pi_G$  below (see ~\eqref{piG}), from an~$S$-arithmetic homogeneous space of the reductive group~$G$ to~$\Sh_K(G,X)$. We then introduce probabilities~$\wt{\mu_n}'$ which
are ``lifts'' of the probabilities~${\mu_n}$, that is such that we have the compatibility~\eqref{lem compat} by direct image.

\subsubsection{The~$S$-arithmetic map}
For convenience we identify the sub\-group~$G_S\subseteq G(\AAA_f)$ with its image~$G(\QQ_S)$. Let us consider the ``orbit maps''
\begin{align*}
 \omega_{tK}&:	G_S\xrightarrow{g\mapsto g\cdot t K} G(\AAA_f)/K&&\text{ at the coset }tK,\\
 \text{ and }~
 \omega_h&:G(\RR) \xrightarrow{g \mapsto g \cdot h} X&&\text{ at the point~$h$.}
\end{align*}
Together these induce the following map
\[
 	\ol{\omega_{(h,tK)}}:
 		G(\RR\times \QQ_S)
 			\xrightarrow{ \omega_{h} \times \omega_{tK}}
		X\times G(\AAA_f)
			\xrightarrow{(h,t)\mapsto \ol{(h,t)}}
		\Sh_K(G,X).
\]
Equivalently, for any element~$(g_\RR,g_S)\in G(\RR)\times G(\QQ_S)\simeq G(\RR\times \QQ_S)$,
\[\ol{\omega_{(h,tK)}}(g_\RR,g_S)=\ol{( g_\RR\cdot h, g_S\cdot t)}.\]
\begin{quote}
\item
\paragraph*{Remark}
The Shimura variety~$\Sh_K(G,X)$ has finitely many geometrically connected components.
Each component is a quotient of a the hermitian symmetric domain~$X^+$ by an arithmetic subgroup of~$G(\QQ)$. The image of~$\ol{\omega_{(h,tK)}}$ 
consists of the union the components this image intersects.\footnote{This is an instance of a ``$S$-adic packet'' of components extracted out of the ``adelic packet'' of all components. For instance, if the Shimura variety is reduced to a class group (and hence finite), viewed as a Galois group, we obtain cosets of the subgroup generated by the Frobenius  from the places in~$S$.}
\end{quote}

The product~$G_S\cdot tKt^{-1}$  is open in~$G(\AAA_f)$. As~$G_S$ is a normal subgroup in~$G(\AAA_f)$, this product is a subgroup. 
We define then following $S$-arithmetic subgroup of $G(\QQ)$
\footnote{We understand ``$S$-arithmetic'' in the sense that there is a faithful $\QQ$-linear representation~$\rho:G\hookrightarrow \GL(n)$, such that the groups~$\rho(G)(\QQ)\cap \GL\left(n,\ZZ[(1/\ell)_{\ell\in S}]\right)$ and~$\rho(\Gamma)$ are commensurable. For~$S=\emptyset$ we recover the usual notion of arithmetic subgroup.
\newline \indent
We will define~$\Gamma$ in~\eqref{defi Gamma}, as an~$S$-arithmetic subgroup of~$G^\der$.}:
\begin{equation}\label{defi Gamma G}
\Gamma_G := G(\QQ)\cap (G_S\cdot tKt^{-1})
\end{equation}
Equivalently we may define, as the intersection, inside of~$G(\AAA)$,
\[
\Gamma_G=\left( G(\RR)\cdot G_S \cdot tKt^{-1} \right)\cap G(\QQ).
\]
This~$\Gamma_G$ depends on~$tK$ and~$S$ though we don't specify it to simplify notation.

The map~$\ol{\omega_{(h,tK)}}$ is left invariant under the action of~$\Gamma_G$: this map factors through a map
\begin{equation}\label{piG}
\pi_G=\pi_{(h,tK)} : \Gamma_G\backslash G(\RR \times \QQ_S) \lto \Sh_K(G,X).
\end{equation}

Let~$K_h$ be the stabiliser of~$h$ in~$G(\RR)$ and~$K_S=K\cap G_S$, so that~$tK_St^{-1}$ is the stabiliser of the coset~$tK$ in~$G_S$.
We may further factor~$\pi_{(h,tK)}$ as the composition
\begin{itemize}
\item of the quotient map
\[
\Gamma_{G} \backslash G(\RR \times \QQ_S)
\to
\Gamma_{G} \backslash \left( G(\RR)\times G(\QQ_S)\right)/\left( K_h\times tK_St^{-1}\right)
\] by the right action of the compact group $K_h \times t K_S t^{-1}$;
\item followed by a closed open immersion into~$\Sh_K(G,X)$, that is  the inclusion of a union of components.
\end{itemize}

\subsubsection{Lift of probabilities}
Recall from~\eqref{defimonogroup} of~Definition~\ref{defiintro} that we denote~$U_S$ denote the~$S$-adic monodromy group associated with~$\overline{(h,t)}$. 
Since $U_S$ is a compact group, it supports a Haar probability.
Let~$\mu_{U_S}$ be \emph{the direct image} of this Haar probability 
in the $S$-arithmetic quotient
\[
	U_S\hookrightarrow G(\RR\times \QQ_S)\to\Gamma_G\backslash G(\RR\times \QQ_S).
\]

The right translate of~$\mu_{U_S}$ by~$g_n$ is denoted by
\begin{equation}\label{lifted1}
	\wt{\mu_n}^\prime=\mu_{U_S}\cdot g_n.
\end{equation}
We now show that the~$\wt{\mu_n}'$ are lifts of the~$\mu_n$.

\begin{lem} Pushing forward the measures~$\wt{\mu_n}'$ from~\eqref{lifted1} along the map~$\pi_G$ from~\eqref{piG} we get
\begin{equation}\label{lem compat}
{\pi_G}_\star(\wt{\mu_n}')=\mu_n.
\end{equation}
\end{lem}
\begin{proof}
The map~$\alpha$ defined by the commutativity of the diagram

 \begin{equation}\label{composee}
 \xymatrix{
	\Gal\left(\ol{E}/E\right)
		\ar[r]^(.65){\rho}
		\ar[drrr]_\alpha
	&
	U_S\ar[drr]^\beta
		\ar[r]^(.30){\sigma\mapsto \Gamma_G\cdot\sigma}
	&
	\Gamma_G\backslash G(\RR\times \QQ_S)
		\ar[r]^{x\mapsto x \cdot g_n}\ar[dr]^{\gamma}
	&
	\Gamma_G\backslash G(\RR\times \QQ_S)
		\ar[d]^{\pi_G}
	\\&&&
	\Sh_K(G,X)
}
\end{equation}
sends~$\sigma$ to~$\ol{(h,\rho(\sigma) t g_n)}$. By the defining properties~\eqref{Galois carac} of~$\rho$ and of~$g_n$, we may rewrite
\(
\alpha(\sigma)=\sigma\cdot \ol{(h,t g_n)}=\sigma\cdot s_n.
\) In other words~$\alpha$ is the orbit map at~$s_n$ for the action of~$\Gal\left(\ol{E}/E\right)$ on~$\Sh_K(G,X)$. In particular its image is the Galois orbit of~$s_n$, which is~$\Supp(\mu_{s_n})$ in the notation of~\eqref{defi mesure}. The map~$\alpha$ is clearly left~$\Gal\left(\ol{E}/E\right)$-equivariant.
Consequently the direct image of the Haar probability, say~$\mu_E$, on~$\Gal\left(\ol{E}/E\right)$ is an invariant probability on~$\Supp(\mu_{s_n})$, necessarily the Haar probability of the transitive~$\Gal\left(\ol{E}/E\right)$-space~$\Supp(\mu_{s_n})$.
The counting probability~$\mu_n=\mu_{s_n}$ is invariant by permutation, hence is~$\Gal\left(\ol{E}/E\right)$-invariant.
It must equal~$\mu_n$ by the uniqueness of the Haar probability:
\[
	\alpha_\star\left(\mu_E\right)=\mu_n.
\]

As~\eqref{composee} is commutative and pushforwards are functorial, we may factor
\[
	\beta_\star\circ \rho_\star=(\beta\circ \rho)_\star=\alpha_\star.
\]
The representation~$\rho$ is a continuous map of compact groups, 
so the direct image~$\rho_\star(\mu_E)$ is the Haar measure on
the image~$\rho(\Gal\left(\ol{E}/E\right))$. This image is~$U_S$
by definition.
It follows that~$\mu_n$ is the direct image of the Haar probability measure~$\rho_\star(\mu_E)$ on~$U_S$ through this~$\beta$. 

We defined~$\mu_{U_S}$ as
the image of the Haar measure of~$U_S$ in the first occurrence in~\eqref{composee} of~$\Gamma_G\backslash G(\RR\times \QQ_S)$, and~$\wt{\mu_n}'$ as the direct image in the second occurrence. By the same functoriality argument as above, the compatibilities
\[
\mu_n=\gamma_\star(\mu_{U_S})={\pi_G}_\star(\wt{\mu_n}')=\beta_\star(\rho_\star(\mu_E))=\alpha_\star(\mu_E),
\]
including the identity~\eqref{lem compat}, follow.
\end{proof}


\subsection{Passing to the derived subgroup}\label{derived}

Theorems from ~\cite{RZ} require that the group $G$ be semisimple. 
We will reduce to this case by passing from~$G$ to~$G^{der}$.
In this section,
 we modify our lifting probabilities~$\wt{\mu_n}'=\mu_{U_S}\cdot g_n$ in two ways
 in order to be able to make this assumption and thus recover the setting of ~\cite{RZ}.
 
\begin{itemize}
\item Firstly we substitute the translating element~$g_n$ in~\eqref{lifted1} with another one which comes from the derived group~$G^\der(\QQ_S)$, thus constructing 
a probability~$\wt{\mu_n}''$. This first step might require passing to a subsequence and altering~$t$.

\item Secondly, we replace the compact subgroup~$U_S$ of~$G(\QQ_S)$ by a compact subgroup~$\Omega$ of~$G^\der(\QQ_S)$, thus producing~$\wt{\mu_n}'''$. This step may require passing to a finite extension of~$E$.
\end{itemize}


We recall that we beforehand ensured that~$U_S\cap G^{\der}(\QQ_S)$ is $\QQ_S$-Zariski connected.

As, for simplicity reasons, the reference~\cite{RZ} deals semisimple groups~$G$, instead of general reductive, we need to carry out the reduction to the semisimple case. A reader uninterested in subtle technical details may skip directly to the next section~\ref{sous section stabilite}.

\subsubsection{Some finite index open subgroups.}
We first note that
\begin{equation}\label{indice fini}
G^{\der}(\QQ_S)\cdot Z(\QQ_S)\text{ is a finite index open subgroup of~$G_S$.}
\end{equation}
\begin{proof}
 As~$S$ is finite, we may work place by place, in which case it is sufficient to refer to~\cite[\S3.2, Cor.~3 p.~122, \S6.4 Cor.~3 p.~320]{PR}.
\end{proof}

In this section we will  use the notation:
\begin{equation}\label{def Gamma Z}
\Gamma_Z=\left(Z(\RR)\cdot \Gamma_G\right)\cap Z(\QQ_S).
\end{equation}
The finiteness of the class group of the torus~$Z$ tells us  that
\[
	(Z(\QQ_S)\cap K) \cdot \Gamma_Z
\]
is an open subgroup of finite index in~$Z(\QQ_S)$. We deduce, combining with~\eqref{indice fini}, that
\begin{equation}\label{indice fini 2}
   G^{\der}(\QQ_S)\cdot	(Z(\QQ_S)\cap K) \cdot \Gamma_Z
\end{equation}
is a finite index open subgroup of~$G_S$.
\subsubsection{Reducing to~$g_n\in G^{\text der.}(\QQ_S)$}
The sequence  of right cosets relative to~\eqref{indice fini 2}, induced by~$(g_n)_{n\geq0}$, that is
\begin{equation}\label{coset gamma Z}
X	\left( g_n\cdot\left(G^{\der}(\QQ_S)\cdot	(Z(\QQ_S)\cap K) \cdot \Gamma_Z\right)\right)_{n\geq 0},
\end{equation}
can, by finiteness of the index below, be decomposed into at most
\begin{equation}
N(G,\Gamma_G)=\left[G^{\text der.}(\QQ_S)\cdot (Z(\QQ_S)\cap K) \cdot \Gamma_Z~:~G_S\right]
\end{equation}
 constant subsequences. (This index only depends on~$G$ and~$\Gamma_G$ as hinted.)

After possibly passing to a subsequence, we may assume the sequence~\eqref{coset gamma Z} has a constant value.
This value is the right coset of~$g_0$. Replacing~$t$ by~$tg_0$, we may assume that~$g_0=1$, and that the right coset of~$g_n$
is the neutral coset. Equivalently,  for every~$n\geq0$,
\begin{equation}\label{decompo dzg}
	\exists (d_n,z_n,\gamma_n)\in G^{\text der.}(\QQ_S)\times	(Z(\QQ_S)\cap K) \times \Gamma_Z,~g_n=d_n\cdot z_n\cdot\gamma_n.
\end{equation}
We will see that we may assume~$z_n=\gamma_n=1$, and replace~$g_n$ by~$d_n$,
without interfering with the compatibility~\eqref{lem compat}. 
To achieve this, similarly to~\eqref{lifted1}, we set:
\begin{equation}\label{lifted2}
	\wt{\mu_n}''=\mu_{U_S}\cdot d_n.
\end{equation}
\begin{lem}\label{lemme compat2} The direct image along~$\pi_G$ of the above~$\wt{\mu_n}''$ is given by
\begin{equation}\label{lem compat2}
{\pi_G}_\star(\wt{\mu_n}'')=\mu_n.
\end{equation}
\end{lem}
\noindent Lemma~\ref{lemme compat2} is deduced from~\eqref{lem compat} using following invariance properties.

\begin{lem}\label{LemmeInvariances}
Let~$\mu$ be any probability measure on~$\Gamma_G\backslash G(\RR)\times G_S$.
\begin{subequations}
\begin{enumerate}
\item \label{item1} For any~$\gamma$ in $\Gamma_G\cap Z(\RR\times\QQ_S)$ we have
\begin{equation}\label{invariance 1}
\mu\cdot\gamma=\mu.
\end{equation}
\item \label{item2} For any~$k\in K_h\times K$, we have
\begin{equation}\label{inva 2}
	{\pi_G}_\star(\mu\cdot k)={\pi_G}_\star(\mu).
\end{equation}
\item \label{item3} For every~$g\in G(\QQ_S)$, and every~$ z\in K\cap Z(\QQ_S)$ or~$z\in Z(\RR)$,
\begin{equation}\label{seconde invariance}
{\pi_G}_\star(\mu\cdot z\cdot g)={\pi_G}_\star(\mu\cdot g).
\end{equation}
\end{enumerate}
\end{subequations}
\end{lem}
\begin{proof}[Proof of \eqref{item1}] Note that~$\Gamma_G\cap Z(\RR\times\QQ_S)$ acts trivially on~$\Gamma_G\backslash G(\RR)\times G_S$,
as can be checked pointwise, for some~$\gamma\in Z(\RR\times\QQ_S)$ and~$\gamma\in\Gamma_S$, with
\[
\Gamma_G \cdot(g_\RR,g_S)\cdot \gamma
=
\Gamma_G \cdot \gamma\cdot (g_\RR,g_S)
=
\Gamma_G \cdot(g_\RR,g_S).
\]
By ``transport of structure'' it acts trivially its measure space.
\end{proof}
\begin{proof}[Proof of \eqref{item2}] 
As~$\pi_G$ factors through the right action of~$K_h\times K$, we have
\[
\forall k\in K\times K_h,~
\pi_G(x\cdot k)=\pi_G(x).
\]
Equivalently~${\pi_G}_\star(\delta_x\cdot k)={\pi_G}_\star(\delta_x)$. Concerning~$\mu$ such as in the statement, we may
compute
\begin{align*}
{\pi_G}_\star(\mu\cdot k)&={\pi_G}_\star\left(\int\delta_x\cdot k~\mu(x)\right)=\int{\pi_G}_\star(\delta_x\cdot k )~\mu(x)\\&
=\int{\pi_G}_\star(\delta_x )~\mu(x)={\pi_G}_\star\left(\int\delta_x~\mu(x)\right)={\pi_G}_\star(\mu).
\end{align*}
We used linearity and continuity of~${\pi_G}_\star$ on bounded measures, seen for instance as continuous linear form on continuous functions.\footnote{We recall that our spaces are `polish' (separable and metrizable) and hence they are Radon spaces: Borel probability measures are inner regular (cf \cite[INT Ch. IX, \S3 \no3 Prop.\,3]{BBKINT}).}
\end{proof}
\begin{proof}[Proof of \eqref{item3}]  As~$z\in Z(\RR\times\QQ_S)$,we may substitute~$z\cdot g=g\cdot z$. Replacing~$\mu$ by~$\mu\cdot g$ we may omit omit~$g$.
We have~$z\in Z(\RR)\leq K_h$ or~$z\in Z(\QQ_S)\cap K\leq K$. In either case we may apply~\eqref{inva 2}.
\end{proof}

\begin{proof}[Proof of Lemma~\ref{lemme compat2}] We note from definition~\eqref{def Gamma Z}
that
\[
\Gamma_Z\subseteq Z(\RR)\cdot \left(\Gamma_G\cap Z(\RR\times\QQ_S)\right).
\]
We may decompose accordingly~\[\gamma_n= \zeta_n\cdot c_n\] with~$\zeta_n\in Z(\RR)$ and~$c_n\in\Gamma_G\cap Z(\RR\times\QQ_S)$. We have
\[
\mu_{U_S}\cdot g_n = \mu_{U_S}\cdot d_n\cdot z_n\cdot \zeta_n\cdot c_n=\mu_{U_S}\cdot d_n\cdot z_n\cdot \zeta_n
\]
where we may omit~$c_n$ in the right-hand side thanks to~\eqref{invariance 1}. Applying~\eqref{seconde invariance}
to~$z=\zeta_n$ and to~$z=z_n$, we deduce
\[
{\pi_G}_\star\left(\mu_{U_S}\cdot d_n\cdot z_n\cdot \zeta_n\right)
=
{\pi_G}_\star\left(\mu_{U_S}\cdot d_n\cdot z_n\right)
=
{\pi_G}_\star\left(\mu_{U_S}\cdot d_n\right).
\]
The Lemma then follows from~\eqref{lem compat}.
\end{proof}

We conclude by noting that 
\[
	\ol{(h,tg_n)}=s_n=\ol{(h,td_n)}.
\]
We may follow the same proof as that of Lemma~\ref{lemme compat2}, but with~$\delta_{(h,tg_n)}$ instead of~$\mu_{U_S}$.
We are now reduced to the case~$g_n\in G^{\text der.}(\QQ_S)$. 

\subsubsection{Passing from~$U_S\leq G_S$ to~$\Omega\leq G^{\der}(\QQ_S)$.}
We now turn to the matter of replacing~$U_S$ by a sugbroup~$\Omega$ of~$G^{\der}(\QQ_S)$.

Let 
\begin{equation}\label{U S tilde}
\wt{U_S}:=U_S\cdot (K\cap Z(\QQ_S)).
\end{equation}
Note that this is a compact group.
We define
\begin{equation}\label{defi Omega}
	\Omega= \wt{U_S}\cap G^{\der}(\QQ_S).
\end{equation}


This~$\Omega$ is a compact group, and therefore carries a Haar probability.
Let~$\mu_\Omega$ be the \emph{direct image} of this Haar probability 
in the $S$-arithmetic quotient
\begin{equation}\label{Sar quotient1}
	\Gamma_G\backslash G(\RR\times \QQ_S).
\end{equation}
Let~
\begin{equation}\label{defi Gamma}
\Gamma=\Gamma_G\cap G^{\der}(\RR\times \QQ_S).
\end{equation}
This is an~$S$-arithmetic group in~$G^{\text der.}(\RR\times \QQ_S)$, and a \emph{lattice} by the $S$-arithmetic form of Borel and Harish-Chandra theorem (e.g.~\cite{GW} by Go\-de\-ment-Weil,
cf~\cite[\S 5.4]{PR}).
We dropped the indiex~in~$\Gamma$, as this will be the~$S$-arithmetic lattice involved when applying~\cite{RZ}, which is denoted by~$\Gamma$ in \emph{loc.\,cit.}

We identify the semisimple arithmetic quotient space
\begin{equation}\label{Sar quotient2}
\Gamma\backslash G^{\text der.}(\RR\times \QQ_S)
\end{equation}
with its image via the natural embedding~$\Gamma\cdot g\mapsto \Gamma_G\cdot g$ into the reductive arithmetic quotient space~\eqref{Sar quotient1}. The support of the probability measure~$\mu_\Omega$ is contained in $\Gamma\backslash G^{\text der.}(\RR\times \QQ_S)$.

We now define
\begin{equation}\label{attempt3}
		\widetilde{\mu_n}'''=\mu_\Omega\cdot g_n,
\end{equation}
as a probability measure on~$\Gamma\backslash G^{\text der.}(\RR\times \QQ_S)$.

Let now
\begin{equation}\label{pider}
	\pi:\Gamma\backslash G^{\text der.}(\RR\times \QQ_S)\to \Sh_K(G,X)
\end{equation}
be the restriction of~$\pi_G$.

We have thus reduced ourselves to the case $G = G^{\der}$.

\subsubsection{Passing to a splitting extension and the lifting property.}
Ensuring that these$~\widetilde{\mu_n}'''$ are still lifts of the~$\mu_n$, as in~\eqref{lem compat2}, will require a bit of extra work.

Recall from~\eqref{indice fini} that~$G^{\text der.}(\QQ_S)\cdot Z(\QQ_S)$ is 
an open subgroup of~$G_S$. 
By its construction in~\eqref{U S tilde}, the group~$\wt{U_S}$ contains an open subgroup of~$Z(\QQ_S)$.
The Lie algebra of~$\wt{U_S}$ hence contains that of~$Z(\QQ_S)$, and is then the sum of the latter with the Lie algebra of~$\Omega$.
Equivalently, the product
\begin{equation}\label{open galois}
	\left(\wt{U_S}\cap Z(\QQ_S)\cap K'\right)\cdot\Omega
\end{equation}
is an open subgroup of~$\wt{U_S}$; which has finite index, as~$U_S$ is compact.
After replacing~$E$ by the finite extension corresponding to this subgroup, we may assume that
\begin{equation}\label{eq 37}
	\wt{U_S}=\left(\wt{U_S}\cap Z(\QQ_S)\cap K'\right)\cdot\Omega,
\end{equation}
which implies
\begin{equation}
	\wt{U_S}\cap Z(\QQ_S)\subseteq Z(\QQ_S)\cap K'\subseteq Z(\QQ_S)\cap K.
\end{equation}
Thus replacing~$E$, we did not change~$\Omega$, nor the associated~$\widetilde{\mu_n}'''$.
We can now prove:


\begin{lem}\label{lemmelift}
The direct images of the measures~$\widetilde{\mu_n}'''$ from~\eqref{attempt3} above, along the map~$\pi$ from~\eqref{pider}, are given by
\begin{equation}\label{lifted3}
\forall n\geq 0,~\pi_\star(\widetilde{\mu_n}''')=\mu_n.
\end{equation}
\end{lem}
This follows from~\eqref{lifted2} and the following equality, proven below,

\begin{equation}\label{lift compat}
\pi_\star(\widetilde{\mu_n}''')={\pi_G}_{\star}(\wt{\mu_n}'').
\end{equation}

\begin{proof}[Proof of~\eqref{lift compat}] Let~$U=\wt{U_S}\cap Z(\QQ_S)$ . Note that the map
\[
\Omega\times U\to \Omega\cdot U=\wt{U_S}.
\]
is a continuous map of compact groups.
It is surjective iun view of~\eqref{eq 37}. The image of the Haar probability measure is a probability which is invariant
under the image of the map. This is hence the Haar probability measure on~$\wt{U_S}$. The direct image measure is actually a convolution of measure. Pushing 
into~$\Gamma_G\backslash G(\RR\times \QQ_S)$, this convolution, in integral form, is defined as
\[
\mu_{\wt{U_S}}=\int_{u\in U} \mu_\Omega\cdot u~du\quad\text{on } \Gamma_G\backslash G(\RR\times \QQ_S)
\]
where the differential notation~$du$ denotes Haar probability measure on~$U$, and~$\mu_{\wt{U_S}}$ is the direct image of the Haar probability on~$\wt{U_S}$ in the~$S$-arithmetic space of~$G^\der$.
Similarly we can prove
\[
\mu_{\wt{U_S}}=\int_{k\in K\cap Z(\QQ_S)} \mu_{U_S}\cdot k~dk
\]
From~\eqref{item3} of Lemma~\ref{LemmeInvariances}, we get, with $dz$ the Haar probability on~$K\cap Z(\QQ_S)$,
\begin{eqnarray*}
{\pi_G}_\star(\mu_{\wt{U_S}}\cdot g_n)
&=&{\pi_G}_\star\left(\left(\int_{z\in K'\cap Z(\QQ_S)} \mu_{U_S}\cdot z~dz\right)\cdot g_n\right)\\
&=&{\pi_G}_\star\left(\int_{z\in K'\cap Z(\QQ_S)} \mu_{U_S}\cdot z\cdot g_n~dz\right)\\
&=&\int_{z\in K'\cap Z(\QQ_S)} {\pi_G}_\star(\mu_{U_S}\cdot z\cdot g_n)~dz\\
&=&\int_{z\in K'\cap Z(\QQ_S)} {\pi_G}_\star(\mu_{U_S}\cdot g_n)~dz\\
&=&{\pi_G}_\star(\mu_{U_S}\cdot g_n)\\
\end{eqnarray*}
In the same manner, we prove
\begin{equation}
{\pi_G}_\star(\mu_{\wt{U_S}}\cdot g_n)
=
{\pi_G}_\star(\mu_{\Omega}\cdot g_n).
\end{equation}
Finally, using Definitions~\eqref{attempt3}, \eqref{pider} we prove~\eqref{lift compat}
\begin{equation}
\pi_\star(\wt{\mu_n}''')={\pi_G}_\star(\mu_{\Omega}\cdot g_n)={\pi_G}_\star(\mu_{\wt{U_S}}\cdot g_n)={\pi_G}_\star(\mu_{U_S}\cdot g_n).\qedhere
\end{equation}
\end{proof}


\subsubsection{Conclusion.}
We now have a sequence~$(\wt{\mu_n}''')_{n\geq0}$ which lifts the sequence~$(\mu_n)_{n\geq0}$, and is the translate~$\mu_\Omega\cdot g_n$ of a probability~$\mu_\Omega$ coming from~$\Omega\leq G^\der(\QQ_S)$ by elements~$g_n$ from~$G^\der(\QQ_S)$, with respect to the semisimple group~$G^\der$. This lifted setting is the setting of~\cite{RZ}. We will now need to verify the hypothesis of~\emph{loc.~cit.}

\subsection{The analytic stability hypothesis.}\label{sous section stabilite}

In this section we will further alter the
sequence~$(\wt{\mu_n}''')_{n\geq0}$ in order to to be able to apply the results of \cite{RZ} to it. 
 

To apply results of \cite{RZ}, the sequence~$(g_n)_{n\geq 0}$ must satisfy a technical ``analytic stability'' hypothesis of~\cite{RZ}, which very loosely speaking means that the ``direction'' in which the~$g_n$ diverge is ``not too close to that of the centraliser of~$\Omega$''.

Recall that~$M$ denotes the Mumford-Tate group of $s$, and that~$H_S$ is the $\QQ_S$-algebraic envelope of the
$S$-adic monodromy group~$U_S$.

Since $s$ satisfies the $S$-Shafarevich assumption, by Proposition \ref{charSha},
we know that
\[
Z_{G}(M)(\QQ_S)\backslash Z_{G_S}(H_S).
\]
is compact. Then the closed subspace 
\[
Z_{G^{\der}}(M)(\QQ_S)\backslash Z_{G^{\der}_S}(H_S)
\]
is also compact. We may find a compact subset~$C$ of~$Z_{G^{\der}_S}(H_S)$ such that
\begin{equation}\label{eq 43}
Z_{G^{\der}_S}(H_S)=Z_{G^{\der}}(M)(\QQ_S)\cdot C.
\end{equation}

We will denote~$Z_S = Z_{G_S^{der}}(H_S)$.
Hence
\[
	Z_{G^{\der}(\RR\times\QQ_S)}(\Omega)=G^{\der}(\RR)\times Z_S.
\]

Recall that the centraliser $Z_{G^{der}}(M)$ is $\RR$-anisotropic, 
as it is contained in the centraliser of the point~$h\in X$, which is compact subgroup. It is a fortiori a~$\QQ$-anisotropic group.

We are in a trivial instance of Godement compactness criterion. There is a compact subset $\cF \subset  Z_{G^{der}}(M)(\RR\times\QQ_S)$ (a ``fundamental set'') such that
\begin{equation}\label{Godement}
Z_{G^{der}}(M)(\RR\times\QQ_S)=\left(Z_{G^{der}}(M)\cap \Gamma_S\right)\cdot \cF.
\end{equation}
Combining with~\eqref{eq 43} we get
\begin{equation}\label{Godement 2}
Z_S=\left(Z_{G^{der}}(M)\cap \Gamma_S\right)\cdot F\text{ with }F=\cF\cdot C
\end{equation}
a compact subset of~$Z_S$.

By the results of~\cite{RS} and~\cite[Partie~2]{R}, there exists a subset~
\[Y=\{1_\RR\}\times Y_S\subseteq G_S,\] where~$1_\RR$ denotes the neutral element of~$G(\RR)$, such that 
\begin{itemize}
\item we have a ``$S$-adic Mostow decomposition''
\begin{equation}\label{Mostow}
G^{\text{der.}}(\RR\times\QQ_S)=Z_{G^{der}}(\Omega)(\RR\times\QQ_S)\cdot (\{1_\RR\}\times Y_S),
\end{equation}
which at finite places becomes
\begin{equation}\label{Mostow2}
G^{\text{der.}}(\QQ_S)=Z_S\cdot Y_S,
\end{equation}
\item any sequence~$(g_n)_{n\geq0}$ in~$\{1_\RR\}\times Y_S$ satisfies the ``analytic stability'' property\footnote{\label{footnote stab 1}Namely the~$y^{-1}$ in~${Y_S}^{-1}$ are such that for any representation~$\rho:G\to GL(N,\QQ_S)$, there is a constant~$c=c({\rho,\Omega})$ such that for any vector~$v\in{\QQ_S}^N$, the action of~$y^{-1}$ one cannot shrink~$\Omega\cdot V$ uniformly by a factor bigger than~$c$: one has~
\begin{equation}\label{eq footnote}
\sup_{\omega\in\Omega} \Nm{y^{-1}\cdot\omega\cdot v}\geq c\cdot\Nm{v}
\end{equation}
with respect, say, to some standard product norm on~$\QQ_S^N$. 
See footnote~\ref{footnote stab 2}.}
of~\cite{RZ} with respect to~$\Omega$ (for~$y^{-1}$ in~$Y_S$ as we deal with left arithmetic quotient~$\Gamma\backslash G$). Here we used that the $\QQ_S$-algebraic envelope~$H_S\cap G_S^\der$ of~$\Omega$ is reductive.
\end{itemize}
Moreover we may, in~$G^{\der}(\RR\times\QQ_S)$, multiply~$\{1_\RR\}\times Y_S$ on the left by a compact subset of~$G^{\der}(\QQ_S)\times Z_S$ (and on the right by any compact subset of~$G$) and the product will still satisfy previous properties\footnote{\label{footnote stab 2}We refer to footnote~\ref{footnote stab 1}. Let~$F$ resp.~$C$ be a compact subset of~$Z_S$ resp.~$G$. For~$(\gamma,y,f,v)\in C\times Y_S\times F\times {\QQ_S}^N$ we have~$\gamma^{-1}\cdot y^{-1}\cdot f^{-1}\cdot\omega\cdot v=\gamma^{-1}\cdot y^{-1}\cdot\omega\cdot f^{-1}\cdot v$. Let~$c_1$ and~$c_2$ be the maximum of the operator norm of~$\rho(\gamma)$ resp.~$\rho(f)$ for~$\gamma$ in~$C$ resp.~$f$ in~$F$. For instance one has~$\Nm{v}=\Nm{f\cdot (f^{-1} \cdot v)}\leq c_2\cdot \Nm{f^{-1}\cdot v}$.
We have
\[
\Nm{\gamma^{-1} y^{-1} f^{-1} \omega\cdot v}
=
\Nm{\gamma^{-1} y^{-1} \omega \cdot f^{-1} v}
\geq c_1\cdot\Nm{y^{-1} \omega \cdot f^{-1} v}
\geq c_1 c\cdot \Nm{f^{-1} v}\\
\geq c_1 c c_2\cdot\Nm{v}.
\]
This proves~\eqref{eq footnote} of footnote~\ref{footnote stab 1} for $F\cdot Y_S\cdot C$ with the constant~$c_1\cdot c\cdot c_2$.
}. This is the case of the subset
\begin{equation}\label{def Y'}
Y'=F\cdot \left(\{1_\RR\}\times Y_S\right)\cdot Z_{G^{\der}}(M)(\RR)\subseteq G^{\der}(\RR\times\QQ_S).
\end{equation}
Let us now decompose~$g_n$ according to~\eqref{Mostow2}
\[
\forall n\geq0, \exists z_n\in  Z_{G^{der}}(M)(\QQ_S),\exists y_n\in Y_S,~
	g_n= z_n\cdot y_n,
\]
and then decompose~$z_n$ inside~$Z_{G^{der}}(M)(\RR\times\QQ_S)$ with respect to the decomposition~\eqref{Godement 2},
\[
	z_n= \gamma_n\cdot f_n,\text{ with }\gamma_n\in \Gamma_S\cap Z_{G^{der}}(M)(\RR\times\QQ_S)\text{ and }f_n\in F.
\]
We recall from~\eqref{invariance 1} that
\begin{equation}\label{gamma invariance}
\mu_\Omega\cdot \gamma_n=\mu_\Omega.
\end{equation}

Substituting~$g_n=\gamma_n\cdot f_n\cdot y_n$ with~$f_n \cdot y_n$ we may, and will, assume that~$\gamma_n=1$.
But doing so we lose the property that~$g_n\in G_S$, as~$\gamma_n$ may be non  trivial at the real place.
Instead we define~
\begin{equation}\label{def gn'}
g^\prime_n={\gamma_n}^{-1}\cdot g_n\cdot {\gamma_{n,\RR}},
\end{equation}
where~${\gamma_{n,\RR}}$ is the real place factor of~$\gamma_n$.
We then have~$g^\prime_n\in G_S$ as we ensured it be trivial at the real place. Define, in our final attempt to lift the~$\mu_n$,
\begin{equation}
	\wt{\mu_n}=\mu_\Omega\cdot g^{\prime}_n.
\end{equation}
We observe that definition~\eqref{def gn'} agrees with~\eqref{def Y'}: we have
\begin{equation}\label{g' dans Y'}
\forall n\geq 0, g'_n\in Y'.
\end{equation}

\begin{prop}\label{prop compat}
The sequence~$(g^\prime_n)_{n\geq0}$ satisfies the analytic stability hypothesis
required by~\cite[Theorem~3]{RZ}. Furthermore, we have
\begin{equation}\label{final compat}
\pi_\star\left( \mu_\Omega\cdot g^\prime_n\right)=\mu_n.
\end{equation}
\end{prop}

\begin{proof}[Proof of~\ref{final compat}] The analytic stability hypothesis is given by~\eqref{g' dans Y'}. It remains to prove the lifting compatibility~\eqref{final compat}.

We have, using~\eqref{gamma invariance}, and recalling definition~\eqref{attempt3},
\[
\wt{\mu_n}=\mu_\Omega\cdot{\gamma_n}^{-1}\cdot g_n\cdot  {\gamma_{n,\RR}}=\mu_\Omega\cdot g_n\cdot  {\gamma_{n,\RR}}=\wt{\mu_n}'''\cdot{\gamma_{n,\RR}}.
\]
But recall that~${\gamma_{n,\RR}}\in Z_{G^{der}}(M)(\RR)$ a compact group which is \emph{included in the stabiliser~$K_h$ of~$h$ in~$X$}. Similarly to~\eqref{seconde invariance}, we deduce that the measures~$\wt{\mu}_n$ and~$\wt{\mu}_n'''$, though they maybe different
as measures on~$\Gamma_S\backslash G^{\text der.}(\RR\times \QQ_S)$,
will satisfy
\[
{\pi_G}_\star(\wt{\mu}_n)={\pi_G}_\star(\wt{\mu}_n''').
\]
We are done: by recalling Lemma~\ref{lemmelift}, and substituting~$\pi$ for~$\pi_G$, as we are dealing with measures supported on~$\Gamma\backslash G^{\text{der.}}(\RR\times\QQ_S)$.
\end{proof}

\subsection{Invoking the Theorem of Richard-Zamojksi}\label{invocation}
We now are in a position to apply~\cite[Theorem 3]{RZ}.
We summarize here how to invoke~\cite[Theorem 3]{RZ}.

\subsubsection{Reviewing the hypotheses.}

\paragraph{About the ambiant group~$G$ and its $S$-arithmetic quotient}
The semisimple algebraic  group~$G$ of~\cite{RZ} is the derived subgroup~$G^{\text der.}$ of our
initial~$G$.
The finite set of places considered in~\cite{RZ} is our set~$S$
together with the archimedean place. The $S$-arithmetic lattice
is our~$\Gamma$ defined by~\eqref{defi Gamma} and~\eqref{defi Gamma G}. 
This defines the ambient~$S$-arithmetic space~$\Gamma\backslash G^{\text der.}(\RR\times\QQ_S)$
in our notations.

\paragraph{About the piece of orbit~$\Omega$}
The~$\Omega$ of~\cite{RZ} is~$\{1_\RR\}\times\Omega$ here. The image~$U_S$ of the representation of the Galois group is a compact~$S$-adic Lie subgroup of~$M(\QQ_S)$. It follows that~$\Omega$, as defined in~\eqref{U S tilde} and~\eqref{defi Omega}, is a bounded~$S$-adic Lie subgroup. It is Zariski connected 
is indeed a bounded $S$-adic Lie subgroup.
 We postpone the proof that its $\QQ_S$-algebraic envelope~$H$ is a Zariski connected subgroup of~$G_S$ to~\ref{H Omega} below.

\paragraph{About the sequence of translates.}
The measure~$\mu_\Omega$ to be translated fits into the scope of study of~\cite{RZ}.

The translating elements~$g_n$ of~\cite{RZ} are our~$g^\prime_n$.
The ``analytic stability'' hypothesis has been taken care of in the preceding section.

\subsubsection{Applying Theorem of~{\cite{R}}}
As a consequence of this analytic stability hypothesis, we may apply
\cite[Th\'{e}or\`{e}me~1.3, Exp.~VI, p.~121]{R} with $H = H_S$ and $Y_S = \{ y_n \}$ and $f$ the characteristic function of $\Omega$ to the sequence of translated probabilities
\[
	\wt{\mu_n}=\mu_\Omega\cdot g_n^\prime
\]
on the $S$-arithmetic homogeneous space
\begin{equation}\label{Homog space}
	\Gamma\backslash G^\text{der.}(\RR\times\QQ_S).
\end{equation}
Thus the sequence~$(\wt{\mu_n}^\prime)_{n\geq0}$ is tight: 
any subsequence contains a subsequence converging to a probability measure.

In particular, after possibly extracting a subsequence, we may assume the sequence~$(\wt{\mu_n}^\prime)_{n\geq0}$ is tightly convergent.

We may now invoke~\cite[Theorem~3]{RZ}. 

\subsubsection{The envelope of~$\Omega$}\label{H Omega} We recall that the $\QQ_S$-algebraic envelope of~$U_S$, the algebraic monodromy group, is denoted~$H_S$, and that~$H_S$ is a reductive group. Here we will determine the $\QQ_S$-algebraic envelope of~$\Omega$, in terms of that of~$H_S$. We discuss why it is a reductive group, and how to ensure it is Zariski connected.

Firstly the  $\QQ_S$-algebraic envelope~$\widetilde{H}_S$ of~$\widetilde{U}_S$ is the algebraic group generated by~$U_S$ and~$K\cap Z(\QQ_S)$. The $\QQ_S$-algebraic envelope of~$Z(\QQ_S)\cap K$ is a Zariski open subgroup~$\widetilde{Z}$  of~$Z(\QQ_S)$: one has~$Z^0\leq \widetilde{Z}\leq Z$. Then
\[
\widetilde{H}_S=H_S\cdot \widetilde{Z}.
\]

Let~$H$ be the  $\QQ_S$-algebraic envelope~$\Omega$. One has~$H\leq G_S^{\der}$ and~$H\leq \widetilde{H}_S$ as~$\Omega$ is contained in both groups. 

Let~$\mathfrak{u},\wt{\mathfrak{u}},\mathfrak{\omega},\mathfrak{z},\mathfrak{h}_S,\wt{\mathfrak{h}}_S,\mathfrak{h}$ be the Lie algebras of~$U$, $\wt{U}_S$, $\Omega$, $Z$, $H_S$ and~$H$. All sums and direct sums will be sum of linear spaces, and the resulting sums will be Lie algebras. 
We have~$\wt{\mathfrak{u}}={\mathfrak{u}}+{\mathfrak{z}}$. The decomposition~${\mathfrak{g}}={\mathfrak{g}}^\der\oplus {\mathfrak{z}}$ induces~$\wt{\mathfrak{u}}={\mathfrak{\omega}}\oplus{\mathfrak{z}}$. Taking algebraic envelopes of Lie subalgebras is compatible with sums. We get
\begin{equation}\label{Lie sums}
\wt{\mathfrak{h}}_S=\mathfrak{h}_S+\mathfrak{z}=\mathfrak{h}\oplus \mathfrak{z},
\end{equation}
and deduce~$\mathfrak{h}=\wt{\mathfrak{h}}_S\cap \mathfrak{g}^\der$. This establishes that~$H$ is a Zariski open subgroup of~$\widetilde{H}_S\cap G_S^\der$. 

The product map~$(h,z)\mapsto h\cdot z$ is an homomorphism, as~$H$ and~$Z$ commute.
At the level of the groups, we deduce an isogeny
\[
H^0\times Z^0\to {H_S}^0\cdot Z^0.
\]	
Recall that~$H_S$ and~$Z$ are reductive. Hence~${H_S}^0\times Z^0$ is reductive too, and so is its quotient~${H_S}^0\cdot Z^0$. The finite cover~$H^0\times Z^0$ must then be reductive, and so is its direct factor~$H^0$. This establishes that~$H$ is a reductive group.

The Zariski neutral component~$H^0$ of~$H$ is that~$(\widetilde{H}_S\cap G_S^\der)^0$, and we have~$H=H^0$ if and only if				
\[
\Omega\subseteq H^0(\QQ_S)=(\widetilde{H}_S\cap G_S^\der)^0(\QQ_S).
\]
As~$H^0(\QQ_S)$ is open in~$H(\QQ_S)$, there is a neighbourhood~$V$ of~$H^0(\QQ_S)$ in~$G_S$ that doesnot meet~$H(\QQ_S)\smallsetminus H^0(\QQ_S)$. Note that~$H^0$ does only depend on~$H_S$, not on~$\Omega$ or~$K$. Provided~$K$ and~$U_S$ are sufficiently small, we may ensure that~$\Omega\subseteq U_S\cdot K\subseteq V$. In such a case, the group~$H$ is Zariski connected.


\subsection{The equidistribution property}
We apply~\cite[Theorem~3]{RZ} in the setting recalled in~\S\ref{invocation}. 
Then, possibly passing to a subsequence of~$\left(\wt{\mu}_{n}\right)_{n\geq0}$,
there exist
\begin{itemize}
\item a~$\QQ$-subgroup~$L\leq G^{\text der.}$ of $S$-Ratner type (as in Definition~\ref{defi S Ratner class}),
\item and elements~$n_\infty,g_\infty\in G^\text{der.}(\RR\times\QQ_S)$
\end{itemize}
such that the limit 
\[
\lim_{n\to\infty} \wt{\mu}_{n}= \mu_{\Lpp}\!\!\cdot n_\infty \star \nu\cdot g_\infty
\]
is the translate by~$g_\infty$ of the convolution of the translated probability measure~$\mu_{\Lpp}\cdot n_\infty$ by the Haar probability~$\nu$ on~$\Omega$.
In terms of Radon measures on~$\Gamma_S\backslash G^{\text der.}(\RR\times \QQ_S)$,
applied to an arbitrary bounded continuous test function~$f$ in~$C^b(\Gamma_S\backslash G^{\text der.}(\RR\times \QQ_S))$, this means
\[
	\lim_{n\to\infty}
		\int f\wt{\mu}_{n}	
=
	\int\! f\mu_{\Lpp}  n_\infty\star \nu\cdot g_\infty
=
	\int_{\omega\in\Omega}\!
		\int_{l\in\Gamma \Lpp} 
			f (l\cdot n_\infty\cdot \omega\cdot g_\infty)
		~\mu_{\Lpp} (l)
	\nu(\omega).
\]

This proves the equidistribution conclusion~\eqref{TheoremeLimite} of Theorem~\ref{Theoreme1}.

\subsubsection{Getting rid of~$n_\infty$.}

After replacing by a subsequence, all the conclusions
of~\cite[Theorem~3]{RZ} hold. In particular we also know that~$n_\infty$
belongs to the closure of~$(\Gamma\cap N)\cdot (Z_{G_S^\der}(\Omega)\cap N)$.
We claim that 
\begin{equation}\label{eq 56}
(\Gamma\cap N)\cdot (Z_{G_S^\der}(\Omega)\cap N)
\end{equation}
is already a closed subset.
\begin{proof} By the $S$-Shafarevich hypothesis, the quotient
\[
Z_{G_S}(\Omega)/Z_{G_S}(M)
\]
is compact. It follows that for any subgroup~$Z$ of~$Z_{G_S}(\Omega)$
the quotient
\[
Z/(Z\cap Z_{G_S}(M))
\subseteq 
Z_{G_S}(\Omega)/Z_{G_S}(M)
\]
is compact. In particular, for~$Z=Z_N(\Omega)=N\cap Z_{G_S^\der}(\Omega)$, we may write
\begin{equation}\label{eq 56b}
Z_{N}(\Omega)=Z_{N}(M)\cdot C
\end{equation}
for a compact subset~$C$. The closedness of~\eqref{eq 56}, that is of 
\[
(\Gamma\cap N)\cdot (Z_{G_S^\der}(\Omega)\cap N)=(\Gamma\cap N)\cdot (Z_{G_S^\der}(M)\cap N)\cdot C
\]
will follow from that of
\[
(\Gamma\cap N)\cdot (Z_{G_S^\der}(M)\cap N).
\]

Let~$F$ be a finite generating subset of~$M$, as a topological group for the Zariski topology, defined over~$\QQ$. Then the orbit map for the adjoint action
\[
G_S^{\der}\xrightarrow{g \mapsto (gfg^{-1})_{f\in F}}  (G_S^{\der})^F
\]
embeds~$G_S/Z_{G_S^\der}(\Omega)$ as a subvariety of~ the affine variety~$G_S^\der$. (In particular~$G_S/Z_{G_S^\der}(\Omega)$ is quasi-affine). We can linearise this action by choosing a faithful representation~$G_S^\der\to GL(N)$ and embedding~$GL(N)^F$ into~$V={M_N}^F$ a Cartesian power of the corresponding matrix space.

Let~$p$ be a generator of~$\det\mathfrak{l}\subseteq\bigwedge^{\dim L}\mathfrak{g}$, the maximal exterior power of the Lie algebra of~$L$. Then~$N$ is defined as the stabiliser of~$p$ in~$G_S^\der$. The orbit map
\[
G_S^{\der}\xrightarrow{g \mapsto g\cdot p=\bigwedge^{\dim L}(\ad_g)(p)}  \bigwedge^{\dim L}\mathfrak{g}
\]
at~$p$ embeds~$G/N$ into the affine variety~$W\bigwedge^{\dim L}\mathfrak{g}$. 

We deduce, with the product map, an embedding
\begin{equation}\label{eq 57}
G/(Z_{G_S^\der}(\Omega)\cap N)\to V\times W,
\end{equation}
the orbit map at~$((f)_{f\in F},p)$.

By definition~$\Gamma$ stabilises an lattice~$\Lambda$ in the free~$\QQ_S$ module~$V\times W$ defined over~$\QQ$, which is arithmetic, that made of~$\QQ$ rational elements. Up to scaling we may assume~$((f)_{f\in F},p)\in\Lambda$. It follows that
\[
(\Gamma\cap N)\cdot ((f)_{f\in F},p)
\]
which is a subset of~$\Lambda$, is discrete, and in particular closed. It follows that its inverse image in~$G_S^\der$,
which is~\eqref{eq 57} is closed as well.
\end{proof}
We may write~$n_\infty=\gamma\cdot z$ with~$\gamma\in \Gamma\cap N$
and~$z\in Z_{G_S^\der}(\Omega)$. We claim that
\[
\mu_{\Lpp}\cdot \gamma=\mu_{\Lpp},
\]
which is proven below. We view~$\nu$ as a measure on~$G_S^\der$ supported on~$\overline{\Omega}$. (Actually~$\overline{\Omega}=\Omega$). We also have
\[
z\cdot \nu=\nu\cdot z.
\]
So we may rewrite the limit measure
\[
\mu_{\Lpp}\!\!\cdot n_\infty \star \nu\cdot g_\infty
=
\mu_{\Lpp}\!\! \star \nu\cdot z\cdot g_\infty.
\]
Subsituting~$g_\infty=z\cdot g_\infty$ we may assume~$n_\infty=1$.
\begin{proof}[Proof of the claim]
Recall that~$L$ is a normal subgroup of~$N$.
Consequently~$\Gamma\cap L(\RR\times\QQ_S)$ is normal in~$\Gamma\cap N(\RR\times\QQ_S)$.
The subgroup~$L(\RR\times\QQ_S)^+$ is normalised by~$N(\RR\times\QQ_S)$, hence by~$\Gamma\cap N(\RR\times\QQ_S)$. We have seen both factors of
\[
L(\RR\times\QQ_S)^+\cdot\left(\Gamma\cap L(\RR\times\QQ_S)\right)
\]
are normalised by~$\Gamma\cap N(\RR\times\QQ_S)$. As~$\Gamma\cap N(\RR\times\QQ_S)$
acts continuously, it normalises the closure of this product, namely~$\Lpp$.

By definition~$\mu_{\Lpp}$ is the right~$\Lpp$-invariant probability on~$\Gamma\backslash\Gamma\Lpp$. It follows that~$\mu_{\Lpp}\gamma$
is the right~$\gamma^{-1}\Lpp\gamma$-invariant probability on~$\Gamma\backslash\Gamma\Lpp\gamma$. We just have seen~$\gamma^{-1}\Lpp\gamma=\Lpp$. But we also have
\[
\Gamma\backslash\Gamma\Lpp\gamma
=
\Gamma\backslash\Gamma\gamma^{-1}\Lpp\gamma
=
\Gamma\backslash\Gamma\Lpp.
\]
It follows that~$\mu_{\Lpp}\gamma=\mu_{\gamma^{-1}\Lpp\gamma}=\mu_{\Lpp}$.
\end{proof}

\section{Focusing criterion and Internality of equidistribution}\label{SectionRZ}

In this section we prove conclusion~\eqref{TheoremeInner} of Theorem~\ref{Theoreme1}, the inclusion of  supports of  the measures~$\mu_n$ in the support of the limiting measure. This inclusion of supports is required in the reasoning of section~\ref{consequence}.

\begin{lem}
For all $n$ large enough, we have the inclusion
$$
{\rm Supp}(\mu_n) \subset \pi(\Supp(\mu_{\infty}))
$$
of closed subsets in~$\Sh_K(G,X)$.
\end{lem}

\subsection{A special case.}
Let us first treat the case where~$g'_n$ belongs to~$\bigcap_{\omega\in \Omega}\omega \Lpp\omega^{-1}$. This is the typical situation featuring the dynamics explicited by~\cite[Theorem~3]{RZ}. We will see later 
in this section how to reduce ourselves to this case.

We first note that~$\Omega$  is a compact subgroup.  Therefore~$\Omega=\overline{\Omega}$ and
\[\Supp(\nu)=\overline{\Gamma\backslash\Gamma\cdot\Omega}=\Gamma\backslash\Gamma\cdot\overline{\Omega}\text{ and }\bigcap_{\omega\in \Omega}\omega \Lpp\omega^{-1}=\bigcap_{\omega\in \overline{\Omega}}\omega \Lpp\omega^{-1}.\]

For any~$\omega$ in~$\Omega$ we rewrite
\[
\Gamma\cdot \omega \cdot g'_n=\Gamma\cdot(\omega \cdot g'_n\cdot \omega^{-1})\omega\in \Gamma\Lpp\omega.
\]
It follows that
\[
\Supp(\wt{\mu}_n)=\Gamma\backslash\Gamma\Omega\cdot g'_n\subseteq\Gamma\backslash\Gamma\Lpp\Omega=\Supp(\mu_{\Lpp}\star\nu).
\]
It is now enough to observe that
\[
\mu_\infty=\pi_\star(\mu_{\Lpp}\star\nu),\]
 which are canonical measures supported on a union of real weakly $S$-special subvarieties associated with~$L$. 


\subsection{The focusing criterion's factorisation}
We recall that we have a limiting distribution
\[
\lim_{n\to\infty} \wt{\mu}_{n}= \mu_{\Lpp} \star \nu\cdot g_\infty
\]
where $\wt{\mu}_{n}=\mu_\Omega\cdot g'_n$ satisfies
\[
	\mu_n=\pi_\star\left(\wt{\mu_n}\right).
\]
We ensured that~$g'_n$ belongs to~$G^\der(\QQ_S)$ (viewed as a subgroup of~$G(\RR\times\QQ_S)$).

We want to prove that, after possibly extracting a subsequence, 
\[
\Supp({\mu}_n)=\Supp(\pi_\star(\wt{\mu}_n))\text{ is contained in }\Supp({\mu}_\infty)=\Supp(\pi_\star(\wt{\mu}_\infty)).
\]

As we are allowed to extract subsequences we may use the more stringent conclusions of~\cite[Theorem~3]{RZ},
the \emph{focusing criterion}, according to which we  may factor
\begin{equation}\label{decompo}
	g'_n=l_n\cdot f_n\cdot b_n
\end{equation}
where~$l_n\in \bigcap_{\omega\in \Omega} \omega L(\RR\times\QQ_S)\omega^{-1}$,
where~$(b_n)_{n\geq 0}$ is a bounded sequence and where~$f_n\in N\cap Z_{G_S^\der}(\Omega)$.

We note that such a decomposition occurs if and only if it occurs place by place. In particular, as
the real component of~$g'_n$ is the neutral element, we may, and we will, substitute the real components of~$l_n$,
of~$f_n$ and of~$b_n$ by the neutral element and still have a decomposition a above. To summarize: without loss of generality, we may assume 
\[l_n, f_n, b_n\in G^\der_S.\]

Moreover, the $S$-Shafarevich hypothesis implies that~$Z_{G_S^\der}(M)$ is cocompact in~$Z_{G_S^\der}(\Omega)$.
\begin{lem} The subgroup~$Z_{G_S^\der}(M)\cap N$ of~$Z_{G_S^\der}(\Omega)\cap N$ is cocompact.
\end{lem}
\begin{proof}Note that~$Z_{G_S^\der}(\Omega)$, as it is a centraliser, is an algebraic subgroup of~$G_S$. Let~$Z^+$ be the maximal isotropic connected $\QQ_S$-subgroup of~$Z_{G_S^\der}(\Omega)$ (the product of the non compact factors of the factor groups~$Z_{G_S^\der}(\Omega)\cap G(\QQ_p)$). This is a normal subgroup of~$Z_{G_S^\der}(\Omega)$. The isotropic factors are generated by unipotent subgroup and split tori. Hence every regular function on~$Z^+$ which is bounded is actually constant.
It also contained in every cocompact algebraic reductive subgroup of~$Z_{G_S^\der}(\Omega)$ (the isotropic factors are generated by unipotent subgroup and split tori, on there which have no nontrivial bounded map).
\end{proof}
The quotient~$\left.(Z_{G_S^\der}(\Omega)\cap N)\middle/(Z_{G_S^\der}(M)\cap N)\right.$ can be identified, as a group, to~$NZ_{G_S^\der}(M)/Z_{G_S^\der}(M)$.

\subsection{Getting rid of the bounded factor~$b_n$} The bounded factors~$b_n$ are just translating, at the ``infinite level''~$\Gamma\backslash G^\der(\RR\times\QQ_S)$, the situation. They involve no asymptotic dynamical feature, and as they belong to~$G_S$, they will essentially be killed at finite level, modulo~$K$. Here are the details.

Possibly passing to an extracted subsequence, we may assume that the bounded
sequence~$(b_n)_{n\geq0}$ is convergent in~$G^\der(\QQ_S)$, with limit, say,~$b_\infty$. It follows that~$(b_n tK)_{n\geq0}$ is a convergent sequence in the space~$G^\der(\QQ_S)K/K$ (it converges to~$b_\infty tK$). As this quotient is a discrete space the convergent sequence~$(b_n tK)_{n\geq0}$ is actually eventually constant.
Possibly extracting further, we may assume that this is a constant sequence.

Substituting~$t$ with~$b_\infty\cdot t$ we may assume~$b_\infty=1$ in~$G^\der(\QQ_S)$,
and thus~$b_ntK=tK$. For any~$x$ in~$\Gamma\backslash G^\der(\RR\times\QQ_S)$ we have~$\pi(x\cdot b_n)=\pi(x)$. Hence
\[
	\pi_\star(\mu\cdot b_n)
	=
	\pi_\star(\mu)
\]
for every bounded measure~$\mu$ on~$\Gamma\backslash G^\der(\RR\times\QQ_S)$. In particular
\[
	\mu_n
	=
	\pi_\star(\wt{\mu}_n)
	=
	\pi_\star(\mu_\Omega\cdot l_n\cdot f_n\cdot b_n)
	=
	\pi_\star(\mu_\Omega\cdot l_n\cdot f_n).
\]
In order to prove, in~$\Gamma\backslash G^\der(\RR\times\QQ_S)K/K\subseteq \Sh_K(G,X)$
\[
\Supp(\pi_\star(\wt{\mu}_n))=\Supp(\mu_n)\subseteq\Supp(\mu_\infty)
\]
we may substitute~$\wt{\mu}_n=\mu_\Omega\cdot l_n\cdot f_n\cdot b_n$ with~$\mu_\Omega\cdot l_n\cdot f_n$. In other terms we may assume~$b_n=1$.

\subsection{Getting rid of the centralising factor~$f_n$} Actually the element~$n_\infty$, that we got rid of, is related to the~$f_n$. Part of the argumentation here will parallel the one that dealt with~$n_\infty$.

As~$f_n$ belongs to both~$Z_{G^\der(\QQ_S)}(\Omega)$ and the normaliser of~$\Lpp$,
we have
\[
f_nM{f_n}^{-1}
=
\bigcap_{\omega\in\Omega}f_n\omega{f_n}^{-1}\cdot f_n\Lpp{f_n}^{-1}\cdot f_n \omega^{-1}{f_n}^{-1}
=
\bigcap_{\omega\in\Omega}\omega (f_n\Lpp{f_n}^{-1})\omega^{-1}.
\]
 the element~$l'_n=f_n^{-1} \cdot l_n\cdot f_n$ belongs to~$\bigcap_{\omega\in\Omega}\omega\Lpp\omega^{-1}$ as much as~$l_n$ does. And so does~$l''_n=\omega\cdot l'_n\cdot\omega^{-1}$. Let us now rewrite
\[
\Gamma\backslash\Gamma \cdot \omega \cdot l_n\cdot f_n
=
\Gamma\backslash\Gamma \omega\cdot f_n \cdot l'_n
=
\Gamma\backslash\Gamma f_n\cdot \omega \cdot l'_n
=
\Gamma\backslash\Gamma f_n\cdot l''_n\cdot \omega.
\]
We consider the algebraic subgroups of~$G^\der(\RR\times\QQ_S)$ given by 
\begin{subequations}
\begin{eqnarray}
N'&=N\cap Z_{G}(M)\\
M&=\bigcap_{\omega\in\Omega}\omega\Lpp\omega^{-1}.
\end{eqnarray}
\end{subequations}
We note that the coset
\[
\Gamma\backslash\Gamma f_n\cdot l''_n
\]
in~$\Gamma\backslash G^\der(\RR\times\QQ_S)$ belongs to the subspace
\[
\Gamma\backslash\Gamma \cdot N'\simeq (\Gamma\cap N')\backslash N'.
\]
As~$M$ is normalised by~$N'$, we may consider the quotient space
\[
(\Gamma\cap N')\backslash N'\cdot M/M
\]
We may hence consider the double coset
\[
(\Gamma\cap N')\cdot f_n\cdot M.
\]
We claim that this sequence is bounded in~$(\Gamma\cap N')\backslash N'M/M$.
\begin{proof}
By contradiction, assume not. Then this sequence of double cosets diverges to infinity in~$
(\Gamma\cap N')\backslash N'\Lpp/\Lpp$. A fotiori~$(\Gamma\cap N')\cdot f_n\cdot l$ diverges to infinity,
uniformly for every~$l$ in~$\Lpp$.  As
\[
(\Gamma\cap N')\backslash N'\simeq \Gamma\backslash\Gamma \cdot N'\subseteq \Gamma\backslash G^\der(\RR\times\QQ_S)
\]
is a closed immersion, hence proper, the~$\Gamma\cdot f_n\cdot l''_n$ diverges uniformly to infinity
in~$\Gamma\backslash G_S^\der$. As~$\Omega$ is bounded, so does the~$\Gamma\cdot f_n\cdot l''_n\cdot \omega$.
It follows that
\[
\lim\wt{\mu}_n=0,
\]
(strongly on every compact) which is not the case.
\end{proof}
Possibly extracting a subsequence we may assume that~$(\Gamma\cap N')\cdot f_n\cdot M$ is convergent,
with some limit, say~$(\Gamma\cap N') n_\infty M$ with~$n_\infty$ in~$N'$. We may write, inside~$N'$,
\[
{\gamma_n}^{-1}\cdot f_n\cdot \lambda_n=n_\infty \cdot \beta_n
\]
with~$\gamma_n$ in~$\Gamma\cap N'$, with~$\lambda_n$ in~$M$ and~$(\beta_n)_{n\geq0}$ a bounded sequence in~$N'$.
We rearrange
\[
\Gamma\backslash\Gamma \cdot \omega \cdot l_n\cdot f_n=
\Gamma\backslash\Gamma\gamma_n\cdot \omega \cdot l'''_n\cdot \lambda'_n\cdot n_\infty\cdot \beta_n.
\]
where~$\omega\gamma_n=\gamma_n\omega$, where~$l'''_n=\gamma_n^{-1}l_n\gamma_n$ and~$\lambda'_n=(n_\infty\cdot \beta_n)\lambda_n\cdot(n_\infty\cdot \beta_n)^{-1}$.
We may factor~$\Gamma\gamma_n=\Gamma$.

\section{Zariski closedness from the $S$-Mumford-Tate hypothesis.}\label{SectionMT}

We put ourselves in the situation of Seciton~\ref{SectionRZ}. We will
use the $S$-Mumford-Tate property to reach the stronger conclusion~\eqref{TheoremeMT} of Theorem~\ref{Theoreme1}, namely that we obtain actual weakly special subvarieties.

\subsection{Normalisation by rational monodromy}
\begin{prop} Assume that the algebraic monodromy subgroup~$H_{S}$ is 
definable over~$\QQ$, that is of form~$H_S=H_{\QQ_S}$ for 
some~$\QQ$-subgroup~$H$ of~$G$.

Then the subgroup of Ratner class~$L$ in~$G$ is normalised by~$H$.
\end{prop}
\begin{proof}
We assume that the sequence $(\mu_n)$ converges to a measure $\mu_{\infty}$ associated to a $\QQ$-group $L$. 
We know that 
$$
g_n = L^H \cdot (N \cap Z_G(H)) \cdot O(1)
$$
with 
$$
L^H = \cap_{h} h L h^{-1}
$$
and
$$
N \subset N_G(L).
$$
We want to show that $L=L^H$.

Without loss of generality, we may assume that $O(1)=1$.

We have that $N \cap Z_G(H)$ is defined over $\QQ$ and
$$
N\cap Z_G(H) = N \cap Z_G(H) \cap \Gamma \times \cF.
$$ 
As $Z_G(H)$ is $\QQ$-anisotropic, $N\cap Z_G(H)$ is $\QQ$-anisotropic and therefore $\cF$ is compact.

We write 
$$
g_n = l_n f_n
$$
with $f_n \in (N \cap Z_G(H))(\RR \times \QQ_S)$ and we write
$$
f_n = \gamma_n \cdot \phi_n
$$
with $\gamma_n \in Z_G(H) \cap \Gamma$ and $\phi_n \in \cF$.

As $\cF$ is compact, after extraction, we may assume that $\phi_n$ is convergent and we may assume that the limit is one.

We can therefore assume that 
$$
g_n = l_n \cdot \gamma_n
$$
and $g_n$ stabilises the closed set $\Gamma \backslash \Gamma L^H \cdot U$ (it is closed because $L^H$ is defined over $\QQ$)

and

$\Gamma \backslash \Gamma L^H U$ contains the support of $\mu_n$.

Therefore $\Gamma \backslash \Gamma L^H U$ contains $Supp (\mu_{\infty})$, hence dim $L^H \geq \dim(Sup(\mu_{\infty})) = \dim(L)$.

With the assumption of the Mumford-Tate, we show that $L$ is normalised by $M =H$.
\end{proof}
Under the $S$-Mumford-Tate type hypothesis we immediately deduce the following.
\begin{cor} Assume moreover that~$s$ is of $S$-Mumford-Tate type, that is~$M=H$.

Then~$L(\RR)$ is normalised by~$h$.
\end{cor}

\subsection{A criterion for a weakly $S$-special real submanifold to be a weakly special subvariety.}

In this section we show that stonger conclusions \ref{TheoremeMT} and
\ref{Theoreme2-2} of the main theorems \ref{Theoreme1} abd \ref{Theoreme2} respectively hold under the assumption that the~$S$-Mumford-Tate hypothesis holds.
These conclusion follow from the following proposition.


\begin{prop} \label{boutonrouge}
Let~$(G,X)$ be a Shimura datum and~$h$ a point in a connected component~$X^+$ of~$X$.
Let~$L \subset G$ be a subgroup such that~$L_{\RR}$ is normalised by~$h(\SSS)$.

Then the image of~$L(\RR)^+ \cdot x \subseteq X^+$ in~$\Gamma\backslash X$ is a weakly special subvariety,
where~$\Gamma$ is any congruence arithmneitc subgroup of~$G$.
\end{prop}

\begin{proof}
First observe that $L_{\RR}$ is normalised by $h(\sqrt{-1})$ which induces a Cartan involution on $G^{ad}_{\RR}$. 
Therefore $L_{\RR}$ is reductive (see \cite[Th.~4.2]{Satake}).

Let $N = N(L)^0$ be the neutral component of the normaliser of $L$ in $G$. 
Because~$L$ is reductive, so are~$N$ and its centraliser~$Z_G(L)$, and we have
$$
N = Z_G(L) \cdot L.
$$
Note that~$h$ factors through $N_{\RR}$.

We have a almost-product decomposition of semisimple groups
$$
N^{der} = Z_G(L)^{der} \cdot L^{der}.
$$
Let~$Z^{c}$ (resp.~$Z^{nc}$) denote the almost product of the almost $\QQ$-factors
of~$Z_G(L)$ which are~$\RR$-compact (resp.\ which are not~$\RR$-compact). Using
the analogous notation for~$L$, we have 
$$
Z_G(L)^{der} = Z^{nc}\cdot Z^c \text{ and } L^{der} = L^{nc}\cdot L^{c}.
$$

By~\cite[Lemme 3.7]{U}, $x$ factors through $H = Z(N) Z^{nc} L^{nc}$.

Let $X_H = H(\RR)\cdot h$.
By~\cite[Lemme 3.3]{U}, $(H,X_H)$ is a Shimura subdatum of $(G,X)$.

Note that
$$
H^{ad} = Z^{nc,ad} \times L^{nc,ad}
$$
and 
write $h^{ad} = (h_1, h_2)$  in this decomposition.

We have
$$
X_H^{ad} = X_1 \times X_2
$$
where $X_1 =  Z^{nc,ad}(\RR)\cdot h_1$ and $X_2 =  L^{nc,ad}(\RR) \cdot h_2$.

By Lemme 3.3 of \cite{U}, both $(Z^{nc,ad},X_1)$ and $(L^{nc,ad},X_2)$ are Shimura data.

The image of $L(\RR)^+ \cdot h$ in $X_{H^{ad}}$ is $\{ h_1 \} \times X_2$,
as in~\cite[Section~2]{UY3}. This finishes the proof.
\end{proof}

For the sake of completeness, we give an alternative proof of the statement
\cite{boutonrouge}.

\begin{proof}[Alternative proof] The symmetry~$s_h$ of~$X$ at~$x$ is induced 
by the Cartan involution given by the conjugation action of~$h(i)$. But~$h(i)$
normalises~$L(\RR)$, hence normalises its neutral component~$L(\RR)^+$,
from which we get that~$L(\RR)^+\cdot h$ under~$s_h$: it is symmetric at~$h$.

Moreover~$L(\RR)^+$ is normalised by~$h(U(1))$, whose conjugation induces the
complex structure on the tangent~$T_hX$ space of~$X$ at~$h$. We deduce that the tangent space of~$L(\RR)^+\cdot h$ at~$h$ is complex subspace of~$T_hX$.

Notice that~$L$ is normalised by~$lhl^{-1}$ for every~$l\in L(\RR)^+$. By the
argument above,~$L(\RR)^+\cdot h$ is symmetric at every point and has a complex
tangent space at every point: it is a symmetric quasi-complex subspace.

By \cite{Helgason}, this quasi-complex structure is a complex structure:~$L(\RR)^+\cdot h$ is a symmetric holomorphic subvariety.

By a theorem of Baily-Borel~\cite{BailyBorel}, the arithmetic quotients
\[
\left(\Gamma\cap L(\RR)^+\right)\backslash L(\RR)^+\cdot h\text{ and }\Gamma\backslash X
\]
are quasi-projective varieties. And by a theorem of Borel~\cite{BorelExtension},
the embedding into~$\Gamma\backslash X$ is algebraic.

It implies that it is a totally geodesic subvriety in the sense of~\cite{Moonen},
that is a weakly special subvariety in our terminology.
\end{proof}

\end{document}